\documentclass{amsart}
\usepackage{amsmath,amssymb,amsthm,amsfonts}
\usepackage{hyperref}

\usepackage{color}

\newtheorem{lemma}{Lemma}[section]
\newtheorem{theorem}{Theorem}[section]

\newtheorem{remark}{Remark}[section]

\numberwithin{equation}{section}

\newcommand{\dis}{\displaystyle}

\newcommand{\D}{\mathbb{D}}
\newcommand{\R}{\mathbb{R}}

\newcommand{\comml}{[\![}
\newcommand{\commr}{]\!]}

\newcommand{\FB}{\mathbf{B}}

\newcommand{\FM}{\mathbf{M}}
\newcommand{\FP}{\mathbf{P}}
\newcommand{\FL}{\mathbf{L}}
\newcommand{\FI}{\mathbf{I}}

\newcommand{\CD}{\mathcal{D}}
\newcommand{\CE}{\mathcal{E}}

\newcommand{\CN}{\mathcal{N}}

\newcommand{\CT}{\mathcal{T}}

\newcommand{\na}{\nabla}

\newcommand{\al}{\alpha}
\newcommand{\be}{\beta}
\newcommand{\ga}{\gamma}
\newcommand{\om}{\omega}
\newcommand{\la}{\lambda}
\newcommand{\de}{\delta}
\newcommand{\si}{\sigma}
\newcommand{\pa}{\partial}
\newcommand{\ka}{\kappa}
\newcommand{\eps}{\epsilon}

\newcommand{\vth}{\vartheta}

\newcommand{\De}{\Delta}
\newcommand{\Ga}{\Gamma}

\newcommand{\lag}{\langle}
\newcommand{\rag}{\rangle}

\newcommand{\trn}{|\!|\!|}
\newcommand{\coml}{{[\![}}
\newcommand{\comr}{{]\!]}}

\begin{document}

\title[Vlasov-Poisson-Boltzmann System]{The Vlasov-Poisson-Boltzmann System for Soft Potentials}

\author[R.-J. Duan]{Renjun Duan}
%\thanks{This project was partially supported by the Direct Grant 2010/2011.}
\address[RJD]{Department of Mathematics, The Chinese University of Hong Kong,
Shatin, Hong Kong}
\email{rjduan@math.cuhk.edu.hk}

\author[T. Yang]{Tong Yang}
\address[TY]{Department of Mathematics, City University of Hong  Kong,
Kowloon, Hong Kong
and
School of Mathematics and Statistics, Wuhan University, P.R.~China}
\email{matyang@cityu.edu.hk}

\author[H.-J. Zhao]{Huijiang Zhao}
\address[HJZ]{School of Mathematics and Statistics, Wuhan University, P.R.~China}
\email{hhjjzhao@hotmail.com}

%\date{\today}

%\thanks{}

\begin{abstract}
An important physical model describing the dynamics of
dilute weakly ionized plasmas in the collisional kinetic theory
is the Vlasov-Poisson-Boltzmann system for which the plasma responds strongly to the
self-consistent electrostatic force.
This paper is concerned with the electron dynamics of  {kinetic plasmas} in the whole space when the
positive charged ion flow provides a spatially uniform background.  We establish the global existence and optimal convergence rates of solutions near a global Maxwellian to the Cauchy problem on the Vlasov-Poisson-Boltzmann system for  angular cutoff soft potentials with $-2\leq \ga<0$. The
main idea is to introduce a time dependent weight function in the velocity
variable to capture the singularity of the cross-section at zero
relative velocity.
\end{abstract}

\maketitle
\thispagestyle{empty}

\setcounter{tocdepth}{1}
\tableofcontents

\section{Introduction}

%\newpage

\subsection{Problem}

The Vlasov-Poisson-Boltzmann  (called VPB in the sequel for simplicity) system is a physical model describing the motion of
dilute weakly ionized plasmas (e.g., electrons and ions in the case of two-species) in the collisional kinetic theory, where  plasmas respond strongly to the self-consistent electrostatic force. In physics, the ion mass is usually much larger than the electron mass so that the electrons move much faster than the ions. Thus, the ions are often
described by a fixed background $n_{\rm b}(x)$ and only the electrons move rapidly. In this simplified case, the VPB system takes the form of
\begin{eqnarray}
% \nonumber to remove numbering (before each equation)
&& \pa_t f + \xi \cdot \na_x f + \na_x\phi \cdot \na_\xi f =Q(f,f),\label{eq1}\\
&& \De_x \phi =\int_{\R^3} f\,d\xi-1,\quad \phi(x)\to 0\ \text{as}\ |x|\to \infty,\label{eq2}\\
&& f(0,x,\xi)=f_0(x,\xi).\label{eq3}
\end{eqnarray}
Here the unknown $f=f(t,x,\xi)\geq 0$ is the density distribution function of the particles located at
$x=(x_1,x_2,x_3)\in\R^3$ with velocity
$\xi=(\xi_1,\xi_2,\xi_3)\in\R^3$ at time $t\geq 0$. The potential function $\phi=\phi(t,x)$ generating the
self-consistent electric field $\na_x\phi$ in \eqref{eq1} is coupled with
$f(t,x,\xi)$ through the Poisson equation \eqref{eq2}, where $n_{\rm b}(x)\equiv 1$ is chosen to be a unit constant.

The bilinear  collision operator $Q$ acting only on the velocity variable, cf.~\cite{CIP-Book,Gl,Vi-Re},
is defined by
\begin{eqnarray}\label{def.Q}
Q(f,g)=\iint_{\R^3\times {\mathbb{S}}^2} |\xi-\xi_\ast|^\ga q_0(\theta)\Big(f(\xi_\ast')g(\xi')-f(\xi_\ast)g(\xi)\Big)\, d \om d\xi_\ast,
\end{eqnarray}
where $(\xi,\xi_\ast)$ and $(\xi',\xi_\ast')$, denoting  velocities of two particles
before and after their collisions  respectively, satisfy
\begin{eqnarray*}
\xi'=\xi-[(\xi-\xi_\ast)\cdot \om]\om,\ \
\xi_\ast'=\xi_\ast+[(\xi-\xi_\ast)\cdot \om]\om,\ \ \om\in {\mathbb{S}}^2,
\end{eqnarray*}
by the conservation of momentum and energy
\begin{equation*}
   \xi+\xi_\ast=\xi'+\xi_\ast',\quad |\xi|^2+|\xi_\ast|^2=|\xi'|^2+|\xi_\ast'|^2.
\end{equation*}
Note that the identity $|\xi-\xi_\ast|=|\xi'-\xi_\ast'|$ holds.

The function $|\xi-\xi_\ast|^\ga q_0(\theta)$ in \eqref{def.Q} is the cross-section
 depending only on $|\xi-\xi_\ast|$ and $\cos\theta=(\xi-\xi_\ast)\cdot \om/|\xi-\xi_\ast|$, and it satisfies $-3<\ga\leq 1$ and it is
assumed to satisfy the  Grad's angular cutoff assumption $0\leq q_0(\theta)\leq C |\cos \theta|$. The exponent $\gamma$ is determined by the potential of intermolecular forces, which is called the soft potential when $-3<\ga <0$, the Maxwell model when $\ga=0$ and the
hard potential when $0<\ga\leq 1$  including the hard sphere model $\ga=1$, $q_0(\theta)={\rm cont.}|\cos\theta|$. For the soft potential, the case $-2< \ga<0$ is called the moderately soft potential and $-3<\ga<-2$ very soft potential, cf.~\cite{Vi-Re}. Notice that the Coulomb potential coincides with the limit at $\ga=-3$ for which the Boltzmann collision operator should be replaced by the Landau operator under the grazing collisions \cite{Guo-VPL,SG,Vi-Re}.

It is an interesting problem to consider the nonlinear stability and convergence rates of  a spatially homogeneous steady state $\FM$  for the Cauchy problem \eqref{eq1}, \eqref{eq2}, \eqref{eq3} when initial data $f_0$ is sufficiently close to $\FM$ in a certain sense, where
$$
\FM=(2\pi )^{-3/2} e^{-|\xi|^2/2}
$$
is a normalized global Maxwellian. This problem was first solved by Guo \cite{Guo2}  for the VPB system on torus in the case of the hard sphere model. Since then, the case of non hard sphere models, i.e. $-3<\ga<1$, under the angular cutoff assumption, has remained open. In a recent work \cite{DYZ} for the VPB system over the whole space, we studied the problem for the case $0\leq \gamma\leq 1$ that includes both the Maxwell model and general hard potentials. However, the method used in \cite{DYZ} can not be applied  to the soft potentials, mainly because the collision frequency $\nu(\xi)\sim (1+|\xi|)^\ga$ with $\ga<0$ is degenerate in the large-velocity domain.

In this paper we shall study the problem in the soft potential case by further developing the approach in \cite{DYZ} with the following extra ingredients to overcome the specific mathematical difficulties in dealing with soft potentials:
\begin{itemize}
  \item  a new time-velocity-dependent  weight in the form of
 $\exp\{\lag\xi\rag^2 [q+\la/(1+t)^{\vth}]\}$ to capture the dissipation for controlling the velocity growth in the nonlinear term for non hard sphere models, and
  \item  a time-frequency/time weighted method to overcome the large-velocity degeneracy in the energy dissipation.
\end{itemize}
The approach and techniques that this paper together with \cite{DYZ} developed
can be applied to some other collisional kinetic models for the non hard sphere interaction potential when an external forcing is present.

\subsection{Main results}
To this end, as in \cite{Guo2},
set the perturbation $u=u(t,x,\xi)$ by
$$
f(t,x,\xi)-\FM = \FM^{1/2} u(t,x,\xi).
$$
Then, the Cauchy problem \eqref{eq1}-\eqref{eq3} of the VPB system is reformulated as
\begin{eqnarray}
% \nonumber to remove numbering (before each equation)
&& \pa_t u + \xi \cdot \na_x u+ \na_x\phi \cdot \na_\xi u -\frac{1}{2} \xi \cdot \na_x\phi u-\na_x \phi \cdot \xi \FM^{1/2} =\FL u +\Ga(u,u),\label{vpb1}\\
&& \De_x \phi=\int_{\R^3} \FM^{1/2} u \,d\xi, \quad \phi(x)\to 0\ \text{as}\ |x|\to \infty,\label{vpb2}\\
&& u(0,x,\xi)=u_0(x,\xi)=\FM^{-1/2} (f_0-\FM),\label{vpb3}
\end{eqnarray}
where the linearized collision operator $\FL$ and the quadratic nonlinear term $\Ga$ are defined by
\begin{eqnarray*}
% \nonumber to remove numbering (before each equation)
&&\FL u ={\FM^{-\frac 12}} \left\{Q\left(\FM, \FM^{1/2} u\right)+ Q\left(\FM^{1/2} u , \FM\right)\right\},\\
&& \Ga(u,u)={\FM^{-\frac 12}} Q\left(\FM^{1/2}u,\FM^{1/2}u\right),
\end{eqnarray*}
respectively. Here, notice that due to \eqref{vpb2}, $\phi$ can be determined in terms of $u$ by
\begin{equation}\label{def.phi}
    \phi(t,x)=-\frac{1}{4\pi |x|}\ast_x\int_{\R^3} \FM^{1/2} u(t,x,\xi) \,d\xi.
\end{equation}
By plugging the above formula into the  equation \eqref{vpb1} for the reformulated VPB system, one has a single evolution equation for the perturbation $u$, cf.~\cite{DY-09VPB,DS-VPB}.

Before stating our main result, we first introduce the following mixed time-velocity weight function
\begin{equation}\label{def.w}
    w_{\tau,q} (t,\xi)=\lag \xi \rag ^{\ga \tau} e^{\lag \xi \rag^2 \left[q+\frac{\la}{(1+t)^\vth}\right]},
\end{equation}
where $\tau\in \R$, $0\leq q\ll 1$, $0<\la\ll 1$, $0<\vth\leq 1/4$, and $\lag \xi \rag =(1+|\xi|^2)^{1/2}$. Notice that even though $w_{\tau,q} (t,\xi)$ depends also on parameters $\la$ and $\vth$, we skip them for notational simplicity,  and in many places we also use $q(t)$ to denote $q+\frac{\la}{(1+t)^{\vth}}$ for brevity.
For given $u(t,x,\xi)$ with the corresponding function $\phi(t,x)$ given by \eqref{def.phi}, we then define a temporal energy norm
\begin{equation}\label{def.tri}
    \trn u \trn_{N,\ell,q} (t)=\sum_{|\al|+|\be|\leq N}\left\|w_{|\be|-\ell,q}(t,\xi)\pa_\be^\al u(t)\right\|+\|\na_x \phi(t)\|_{H^N},
\end{equation}
where $N\geq 0$ is an integer, and $\ell\geq N$ is a constant.

The main result of this paper is stated as follows. Some more notations will be explained at the end of this section.

\begin{theorem}\label{thm.m}
Let $-2\leq \ga<0$, $N\geq 8$,  $\ell_0>\frac{5}{2}$,  {$\ell\geq 1+ \max\left\{N, \frac{\ell_0}{2}-\frac1\ga\right\}$} and $q(t)=q+\frac{\la}{(1+t)^{\vth}} \geq 0$ with $0\leq q\ll 1$, $0<\la \ll 1$, and $0<\vth\leq 1/4$.
Assume that $f_0=\FM +\FM^{1/2}u_0\geq 0$ and
\begin{equation}\label{ass.neu}
  \iint_{\R^3\times \R^3} \FM^{1/2} u_0\,dxd\xi=0.
\end{equation}
There are constants $\eps_0>0$, $C_0>0$
such that if
\begin{equation}\label{thm.ge.1}
    \sum_{|\al|+|\be|\leq N}\left\|{w_{|\beta|-\ell,q}}(0,\xi) \pa_\be^\al u_0\right\|+\left\|\left(1+|x|+|\xi|^{-\frac {\ga \ell_0}{2}}\right)u_0\right\|_{Z_1}\leq \eps_0,
\end{equation}
then the Cauchy problem \eqref{vpb1}, \eqref{vpb2}, \eqref{vpb3} of the VPB system admits a unique global solution $u(t,x,\xi)$ satisfying $f(t,x,\xi)=\FM+\FM^{1/2}u(t,x,\xi)\geq 0$ and
\begin{multline}\label{thm.ge.2}
   \sup_{t\geq 0} \left\{{\trn u\trn_{N,\ell,q} (t)+(1+t)^{\frac{3}{4}} \trn u \trn_{N,\ell-1,q} (t)}\right.\\
   \left.+(1+t)^{\frac{5}{4}}\left\|\na_x^2\phi (t)\right\|_{H^{N-1}}\right\} \leq C_0 \eps_0.
\end{multline}
\end{theorem}

\begin{remark}
The condition \eqref{ass.neu} implies that the ionized plasma system with electrons and ions are neutral at initial time. Due to the conservation of mass for \eqref{eq1} or \eqref{vpb1}, the neutral condition \eqref{ass.neu} holds for all time $t>0$. Notice that \eqref{ass.neu} together with $\|(1+|x|)u_0\|_{Z_1}<\infty$  can be used to remove the singularity of the Poisson kernel and to induce the same time-decay rate as in the case of the  Boltzmann equation; see also Remark \ref{rem.lide}. Moreover, in \eqref{thm.ge.1}, we put the extra space-velocity weight $1+|x|+|\xi|^{-\frac {\ga\ell_0}{2}}$ on $u_0$ in $Z_1$-norm. This is necessary in the proof of Lemma \ref{lem.X} concerning the desired uniform-in-time {\it a priori} estimate.
\end{remark}

\begin{remark}
The problem of the global existence for the very soft potential case $-3<\ga<-2$ is still left open. Despite this, we shall discuss in the next subsections the main difficulty.
\end{remark}

We conclude this subsection by pointing out some possible extensions of Theorem \ref{thm.m}. First,
the arguments used in this paper can be adopted straightforwardly to deal with either the VPB system with soft potentials on torus with additional conservation laws as in \cite{Guo2} or the two-species VPB system with soft potentials as in \cite{Guo1,YY-CMP}. Next, it could be quite interesting to apply the current approach as well as techniques in \cite{Guo-L,SG} to consider the Vlasov-Poisson-Landau system \cite{Guo-VPL},  where $Q$ in \eqref{eq1} is replaced by the Landau operator
\begin{equation*}
    Q^{L}(f,f)=\na_\xi\cdot \left\{\int_{\R^3}B^L(\xi-\xi_\ast)
    [f(\xi_\ast)\na_\xi f(\xi)-f(\xi)\na_\xi f(\xi_\ast)]\,d\xi_\ast\right\},
\end{equation*}
where $B^L(\xi)$  is a non-negative matrix given by
\begin{equation*}
    B^L_{ij}(\xi)=\left(\de_{ij}-\frac{\xi_i\xi_j}{|\xi|^2}\right)|\xi|^{\ga+2},\
    \ \ga\geq -3.
\end{equation*}
Note that the soft potential for the Landau operator corresponds to the case $-3\leq \ga<-2$.

\subsection{Strategy of the proof}
We first recall several elementary properties of the linearized operator $\FL$. First of all, $\FL$ can be written as  $\FL=-\nu+ K$, where
\begin{equation*}
   \nu=\nu(\xi)\sim (1+|\xi|)^\ga
\end{equation*}
denotes the collision frequency, and $K$ is a velocity integral operator with a real symmetric integral kernel $K(\xi,\xi_\ast)$. The explicit representation of $\nu$ and $K$ will be given in the next section.
It is known that $\FL$ is non-positive, the null space of $\FL$ is given by
\begin{equation*}
    \CN={\rm span}\left\{\FM^{1/2},\xi_i\FM^{1/2}~(1\leq i\leq 3),
    |\xi|^2\FM^{1/2} \right\},
\end{equation*}
and $-\FL$ is coercive in
the sense that there is a constant $\kappa_0>0$ such that, cf.~\cite{CIP-Book,Guo-BE-s,Mo}
\begin{equation}\label{coerc}
    -\int_{\R^3}u\FL u\,d\xi\geq \ka_0\int_{\R^3}\nu(\xi)|\{\FI-\FP\}u|^2d\xi
\end{equation}
holds for $u=u(\xi)$, where $\FI$ means the identity operator and $\FP$ denotes the orthogonal projection from
$L^2_\xi$ to $\CN$. As in \cite{Guo-IUMJ,Guo-L},  for any
given any $u(t,x,\xi)$, one can write
\begin{equation*}
\left\{\begin{split}
\dis &  \FP u= \left\{a(t,x)+b(t,x)\cdot \xi+c(t,x)\left(|\xi|^2-3\right)\right\}\FM^{1/2},\\
  \dis   & a= \int_{\R^3} \FM^{1/2} u\,d\xi= \int_{\R^3} \FM^{1/2} \FP u\,d\xi,
  \\
  \dis  & b_i=\int_{\R^3}\xi_i \FM^{1/2}u\,d\xi
=\int_{\R^3} \xi_i \FM^{1/2}\FP u\,d\xi,\ \ 1\leq i\leq 3,
\\
  \dis & c= \frac{1}{6}\int_{\R^3}\left(|\xi|^2-3\right) \FM^{1/2} u\,d\xi
= \frac{1}{6}\int_{\R^3} \left(|\xi|^2-3\right) \FM^{1/2} \FP u\,d\xi.
    \end{split}\right.
\end{equation*}
so that we have the macro-micro decomposition introduced in \cite{Guo-IUMJ}
\begin{equation}\label{mm.decom}
    u(t,x,\xi)=\FP u(t,x,\xi)+\{\FI-\FP\} u(t,x,\xi).
\end{equation}
Here, $\FP u$ is called the macroscopic component of $u(t,x,\xi)$ and $\{\FI-\FP\}u$ the
microscopic component of $u(t,x,\xi)$.
%cf.~\cite{Guo2, Guo1,Guo-IUMJ} and \cite{Liu-Yu-Shock,Liu-Yang-Yu}.
For later use, one can rewrite $\FP $ as
\begin{equation*}
\left\{\begin{array}{l}
 \dis \FP u= \FP_0 u\oplus \FP_1 u,\\[3mm]
 \dis \FP_0 u=a(t,x)\FM^{1/2},\quad \\[3mm]
\dis \FP_1u= \left\{b(t,x)\cdot \xi+c(t,x)\left(|\xi|^2-3\right)\right\}\FM^{1/2},
\end{array}\right.
\end{equation*}
where $\FP_0$ and $\FP_1$ are the projectors corresponding to the
hyperbolic and parabolic parts of the macroscopic component,
respectively, cf.~\cite{DS-VPB}.

To prove Theorem \ref{thm.m}, we introduce the equivalent temporal energy functional
\begin{equation}\label{def.e}
    \CE_{N,\ell,q}(t)\sim \trn u \trn_{N,\ell,q}^2 (t),
\end{equation}
and the corresponding dissipation rate functional
\begin{multline}\label{def.ed}
% \nonumber to remove numbering (before each equation)
   \CD_{N,\ell,q}(t) =  \sum_{|\al|+|\be|\leq N}\left\|\nu^{1/2}w_{|\be|-\ell,q}(t,\xi)\pa_\be^\al\{\FI-\FP\} u(t)\right\|^2\\
   +\frac{1}{(1+t)^{1+\vth}} \sum_{|\al|+|\be|\leq N}\left\|\lag \xi \rag w_{|\be|-\ell,q}(t,\xi)\pa_\be^\al\{\FI-\FP\} u(t)\right\|^2\\
   +\|a\|^2+\sum_{|\al|\leq N-1}\left\|\na_x\pa^\al (a,b,c)\right\|^2.
\end{multline}
We also introduce the time-weighted temporal sup-energy
\begin{multline}\label{def.x}
    X_{N,\ell,q}(t)=\sup_{0\leq s\leq t}\CE_{N,\ell,q}(s)+\sup_{0\leq s\leq t} (1+s)^{\frac{3}{2}}\CE_{N,\ell-1,q}(s)\\
    +\sup_{0\leq s\leq t} (1+s)^{\frac{5}{2}}\left\|\na_x^2\phi (s)\right\|_{H^{N-1}}^2.
\end{multline}
Here, the exact  definitions of   $\CE_{N,\ell,q}(t)$  {is given by \eqref{lem.ee.p10}} in the proof of Lemma \ref{lem.ee}.

The strategy to prove Theorem \ref{thm.m} is to obtain the uniform-in-time  estimates under the {\it a priori} assumption that $X_{N,\ell,q}(t)$ is small enough over $0\leq t<T$ for any given $T>0$. Indeed,
\begin{itemize}
  \item
  one can deduce that there is a temporal energy functional $ \CE_{N,\ell,q}(t)$ satisfying \eqref{def.e} such that
  \begin{eqnarray}
  % \nonumber to remove numbering (before each equation)
    &&\frac{d}{dt} \CE_{N,\ell,q}(t)+\kappa \CD_{N,\ell,q}(t)\leq 0,\label{est.e}
  \end{eqnarray}
  for $0\leq t<T$, where $\CD_{N,\ell,q}(t)$ is given by \eqref{def.ed}, and moreover,
  \item by combining the time-decay property of the linearized system, one can prove that
  \begin{equation}\label{est.X}
  X_{N,\ell,q}(t)\leq C\left\{\eps_{N,\ell,q}^2+X_{N,\ell,q}(t)^2\right\}
  \end{equation}
   for $0\leq t<T$, where $\eps_{N,\ell,q}$ depends only on initial data $u_0$.
\end{itemize}
\noindent Therefore, by the local existence of solutions as well as the continuity argument, $ X_{N,\ell,q}(t)$ is bounded uniformly in all time $t\geq 0$ as long as $\eps_{N,\ell,q}$ is sufficiently small, which then implies Theorem \ref{thm.m}.

We remark that this kind of strategy used here was first developed in \cite{DUYZ-CMP} for the study of the time-decay property of the linearized Boltzmann equation with external forces and was later revisited in \cite{DUY} for further generalization of the approach.

\subsection{Difficulties and ideas}
We now outline detailed ideas to carry out the above strategy with emphasize on the specific mathematical difficulties.
Before that, it is worth to pointing out that for the VPB system with soft potentials, the only result available so far is \cite{Guo-Vacuum} in which global classical solutions near vacuum are constructed.
One may expect that the work \cite{Ukai-1974} and its recent improvement \cite{UY-AA} by using the spectral analysis and the contraction mapping principle can be adapted to deal with this problem.  However, we note that when the self-induced potential force is taken into account, even for the hard-sphere interaction, with the spectral property obtained in \cite{GS-DCDS}, the spectral theory corresponding to \cite{Ukai-1974} has not been known so far, partially because the Poisson equation produces an additional nonlocal term with singular kernels.

Fortunately, the energy method recently developed in \cite{Guo-L,Guo-IUMJ} works well  in the presence of the self-induced electric field \cite{Guo2,YZ-CMP,DY-09VPB} or even electromagnetic field \cite{Guo1,St-VMB,D-1VMB}. But due to the appearance of the nonlinear term $\xi\cdot\nabla_x\phi \{{\bf I}-{\bf P}\}u$, the direct application of the coercive estimate \eqref{coerc} of the linearized collision operator ${\bf L}$ works only for the hard-sphere interaction with $\ga=1$. Indeed with a bit weaker (softer) than hard-sphere interaction, such a term is beyond the control by either the usual energy or the energy dissipation rate so that even the local-in-time solutions can not be constructed within this framework.

To overcome this difficulty, a new weighed energy method was introduced in \cite{DYZ} by the authors of this paper to deal with the VPB system with hard potentials. Such a method is based on the use of a mixed time-velocity weight function
\begin{equation}\label{def.w-hard}
    w^{hp}_\ell (t,\xi)=\lag \xi \rag ^{\frac{\ell}{2}} e^{\frac{\la |\xi|}{(1+t)^\vth}},
\end{equation}
where $\ell\in \R$, $\la>0$ and $\vth>0$ are suitably chosen constants. One of the most important ingredients is to combine the time-decay of solutions with the usual weighted energy inequalities in order to obtain the uniform-in-time  estimates.

The main purpose of this paper is  to generalize the above argument so that it can be adapted to deal with the soft potential case. Compared with the hard potential case, the main difference lies in the fact that at high velocity the dissipation is much weaker than the instant energy. Our main ideas to overcome this are explained as follows.

Firstly, similar to the hard potential case \cite{DYZ}, we introduce an exponential weight factor
$$
w^e_{\tau,q} (t,\xi)=\exp\left\{\lag \xi \rag^2 \left[q+\frac{\la}{(1+t)^\vth}\right]\right\}
$$
in the weight function $w_{\tau,q}(t,\xi)$ given by \eqref{def.w}. Notice that $q(t):=q+\la/(1+t)^{\vth}$ satisfies $0<q(t)\ll 1$ by choices of $q$, $\la$ and $\vth$ as stated in Theorem \ref{thm.m}, and hence all the estimates on $\nu$ and $K$ obtained in \cite{SG}  can still be applied with respect to the weight function $w_{\tau,q}(t,\xi)$ considered here. A key observation to use the exponential factor is based on the identity
$$
w^e_{\tau,q} (t,\xi)\pa_t g(t,x,\xi)=\pa_t\left\{w^e_{\tau,q} (t,\xi) g(t,x,\xi)\right\}+\frac{\lambda \vth\lag \xi \rag^2}{(1+t)^{1+\vth}}w^e_{\tau,q} (t,\xi) g(t,x,\xi).
$$
Thus one can gain from the second term on the right hand side of the above identity an additional second-order velocity moment $\lag \xi\rag^2$ with a compensation that the magnitude of the additional dissipation term decays in time with a rate $(1+t)^{-(1+\vth)}$. With this, the nonlinear term $\xi\cdot\nabla_x\phi \{{\bf I}-{\bf P}\}u$ can be controlled as long as the electric field $\na_x\phi$ has the time-decay rate not slower than $(1+t)^{-(1+\vth)}$. Notice that in the case of the whole space, $\|\na_x^2\phi\|_{H^{N-1}}$ as a part of the high-order energy functional decays  at most as $(1+t)^{-\frac{5}{4}}$ and thus it is natural to require $0<\vth\leq 1/4$.

Moreover, as used for hard potentials in \cite{DYZ}, the first-order velocity moment $\lag \xi\rag$ in the exponential factor of \eqref{def.w-hard} is actually enough to deal with the large-velocity growth in $\xi\cdot\nabla_x\phi \{{\bf I}-{\bf P}\}u$. The reason why we have chosen $\lag \xi\rag^2$ for soft potentials is that
one has to control another nonlinear term $\na_x\phi\cdot \na_\xi \{\FI-\FP\} u$. We shall clarify this point in more details later.

The second idea is to control the velocity $\xi$ derivative of $u$. It is well-known that $\nabla_\xi u$ may grow in time. To overcome such a difficulty, we apply the algebraic velocity weight factor, introduced in \cite{Guo-L,Guo-BE-s},
$$
w^v_{|\beta|-\ell} (\xi)=\lag \xi \rag^{\ga(|\beta|-\ell)},\quad \ell\geq |\beta|,
$$
in the weight $w_{\tau,q}(t,\xi)$ given in \eqref{def.w} with $\tau=|\be|-\ell$. This algebraic factor implies two  properties of the velocity weight: one is that $w^v_{|\beta|-\ell} (\xi) \geq 1$ holds true due to $\ell\geq |\be|$, and the other one is that the higher the order of the $\xi$-derivatives, the more negative velocity weights. The restriction $\ell\geq |\be|$, more precisely
$$
\ell-1\geq  \max\left\{N, \frac{\ell_0}{2}-\frac1\ga\right\}\geq |\be|,
$$
results from the fact that one has to use the positive-power algebraic velocity weight $w^v_{|\beta|-\ell} (\xi)$ so as to obtain the closed estimate \eqref{est.X} on $X_{N,\ell,q}(t)$, because not only the nonlinear term $\xi\cdot \na_x\phi u$ contains the first-order velocity growth but also initial data is supposed to have an extra  positive-power velocity weight $\lag \xi\rag^{-\frac{\ga\ell_0}{2}}$ in the time-decay  estimate \eqref{thm.lide.3} for the evolution operator of the linearized VPB system.

As pointed out in \cite{Guo-BE-s} for the study of the Boltzmann equation with soft potentials, the dependence of the weight $w^v_{|\beta|-\ell} (\xi)$ on $|\be|$ makes it possible to control the linear term $\xi\cdot \na_x u$ in terms of the purely $x$-derivative dissipation for the weighted energy estimate on the mixed $x$-$\xi$ derivatives. However, the algebraic velocity factor $w^v_{|\beta|-\ell} (\xi)$ produces an additional difficulty on the nonlinear term $\na_x\phi\cdot \na_\xi u$ in the presence of the self-consistent electric field for the VPB system with soft potentials. To obtain the velocity weighted derivative  estimate on such a nonlinear term, one should put an extra negative-power function $\lag \xi\rag^{\ga}$ in front of $\na_\xi \{\FI-\FP\} u$ in the  nonlinear term $\na_x\phi\cdot \na_\xi u$ so that the velocity growth $\lag \xi\rag^{-\ga}$ comes up to have a balance. Then, as long as
$$
-2\leq \ga<0,
$$
it is fortunate that the dissipation functional $\CD_{N,\ell,q}(t)$ given by \eqref{def.ed} containing the second-order moment $\lag \xi\rag^2$ can be used to control the term $\na_x\phi\cdot \na_\xi u$.

At this point, it is not known how to use  the argument of this paper to deal with the very soft potential case $-3<\ga<-2$.   However, inspired by \cite{Guo-VPL}, we remark that one possible way is to have a stronger dissipation property of the linearized collision operator ${\bf L}$, particularly the possible velocity diffusion effect. In fact, for the Boltzmann equation without angular cutoff assumption, from the coercive estimate on the linearized Boltzmann collision operator ${\bf L}$ obtained in \cite{AMUXX, GrSt}, it seems hopeful to deal with the case $-3<\ga<-2$ for such a physical situation, which will be reported in the future work.

Another ingredient of our analysis is the decay of solutions to the VPB system for soft potentials. Recall that the large-time behavior of global solutions has been studied extensively in recent years by using different approaches. One approach which usually leads to slower time-decay than in the linearized level is used in \cite{YZ-CMP} on the basis of the improved energy estimates together with functional inequalities. The method of thirteen moments and compensation functions proposed by Kawashima in  \cite{Ka-BE13} which gives the optimal time rate without using the spectral theory; see \cite{GS-TTSP} and \cite{YY-CMP} for two applications. Recently, concerning with the optimal time rate, a time-frequency analysis method has been developed in \cite{DS-VPB,DS-VMB,D-1VMB}. Precisely,  in the same spirit of \cite{Vi}, some time-frequency functionals or interactive energy functionals are constructed in \cite{DS-VPB,DS-VMB,D-1VMB} to capture the dissipation of the degenerate components of the full system.

Back to  {the nonlinear VPB system, the main difficulty of deducing the decay rates of solutions for the soft potential} is caused by the lack of a spectral gap for the linearized collision operator ${\bf L}$.  {Unlike the periodic domain \cite{SG}, we need a careful and delicate estimate on the time decay of solutions to the corresponding linearized equation in the case of the whole space $\R^3$. Following the recent work \cite{DS-VPB,DS-VMB}, our analysis} is based on the weighted energy estimates, a time-frequency analysis method, and the construction of some interactive energy functionals, which gives a new and concise proof of the decay of the solution to the linearized equation with soft potentials. Here, we should mention the recent work \cite{St-Op} by Strain. Different from this work, by starting from the inequality
\begin{equation*}
   \pa_t E_\ell(\hat{u})+ \kappa \rho(k) E_{\ell-1}(\hat{u})\leq 0,
\end{equation*}
we use a new time-frequency weighted approach together with the time-frequency splitting technique in order to deduce the time-decay estimate on the solution $u$ for the linearized VPB system; see Step 3 in the proof of Theorem \ref{thm.lide} and the identity \eqref{split.tf}.

Finally we also point out a difference between the soft potential case and hard potentials studied in \cite{DYZ} for obtaining the time-decay of the energy functional $\CE_{N,\ell,q}(t)$. The starting point is the Lyapunov inequality \eqref{est.e}. In the hard potential case, the Gronwall inequality can be directly used when the lowest-order terms $\|(a,b,c)\|^2$ and $\|\na_x\phi\|^2$ are added into $\CD_{N,\ell,q}(t)$. However, this fails for the soft potentials because $\CD_{N,\ell,q}(t)$ given by \eqref{def.ed} contains the extra factor $\nu(\xi)\sim\lag \xi\rag^\ga$ which is degenerate at large-velocity domain when $\ga<0$. To overcome this difficulty,  {we have used the time-weighted estimate on \eqref{est.e} as well as an iterative technique on the basis of the inequality
\begin{equation*}
  \CD_{N,\tilde{\ell},q}(t)+ \|(b,c)(t)\|^2+\|\na_x\phi(t)\|^2\geq   \kappa\CE_{N,\tilde{\ell}-\frac{1}{2},q}(t)
\end{equation*}
for any $\tilde{\ell}$; see Step 2 in the proof of Lemma \ref{lem.X} for details.}

Before concluding this introduction, we mention some previous results concerning the study of the VPB system in other respects, cf.~\cite{DD,GJ}, the Boltzmann equation with soft potentials,  cf.~\cite{UA-s,Ca1,Ca2}, and also the exponential rate for the Boltzmann equation  with general potentials in the collision kernel but under additional conditions on the regularity of solutions, cf.~\cite{DV}.

The rest of this paper is organized as follows. In Section \ref{sec2}, we will list some estimates on $\nu$ and $K$ proved in \cite{SG}. In Section \ref{sec3},  we will study the time-decay property of the linearized VPB system without any source. In Section \ref{sec4}, we give the  estimates on the nonlinear terms with respect to the weight function $w_{\tau,q}(t,\xi)$. In Section \ref{sec5}, we will close the {\it a priori} estimates on the solution to derive the desired inequality \eqref{est.e}. In the last section, we will prove \eqref{est.X} to conclude Theorem \ref{thm.m}.

\subsection{Notations}
Throughout this paper,  $C$  denotes
some positive (generally large) constant and $\ka$ denotes some positive (generally small) constant, where both $C$ and
$\ka$ may take different values in different places. $A\lesssim B$ means $A\leq \frac{1}{\kappa} B$ and
$A\sim B$ means $\ka A\leq B \leq \frac{1}{\ka} A$, both for a generic
constant $0<\ka<1$. For an
integer $m\geq 0$, we use $H^m_{x,\xi}$, $H^m_x$, $H^m_\xi$ to denote the
usual Hilbert spaces $H^m(\R^3_x\times\R^3_\xi)$, $H^m(\R^3_x)$,
$H^m(\R^3_\xi)$, respectively, and $L^2$, $L^2_x$, $L^2_\xi$ are
used for the case when $m=0$. When without any confusion, we use $H^m$ to denote $H^m_x$ and use $L^2$ to denote $L^2_x$
 or $L^2_{x,\xi}$. We denote $\lag \cdot,\cdot \rag$ by the inner product over $L^2_{x,\xi}$.
For $q\geq 1$, we also define the  mixed  velocity-space Lebesgue
space $Z_q=L^2_\xi(L^q_x)=L^2(\R^3_\xi;L^q(\R^3_x))$ with the norm
\begin{equation*}
\|u\|_{Z_q}=\left(\int_{\R^3}\left(\int_{\R^3}
    |u(x,\xi)|^q \,dx\right)^{2/q}d\xi\right)^{1/2},\ \ u=u(x,\xi)\in Z_q.
\end{equation*}
For
multi-indices $\al=(\al_1,\al_2,\al_3)$ and
$\be=(\be_1,\be_2,\be_3)$, we denote $\pa^{\al}_\be=\pa_x^\al\pa_\xi^\be$, that is,
$
 \pa^{\al}_\be=\pa_{x_1}^{\al_1}\pa_{x_2}^{\al_2}\pa_{x_3}^{\al_3}
    \pa_{\xi_1}^{\be_1}\pa_{\xi_2}^{\be_2}\pa_{\xi_3}^{\be_3}.
$
The length of $\al$ is $|\al|=\al_1+\al_2+\al_3$ and the length of
$\be$ is $|\be|=\be_1+\be_2+\be_3$.

%\newpage
\section{Preliminaries}\label{sec2}

Recall that $\FL=-\nu+K$ is defined by
\begin{equation*}
    \nu(\xi)=\iint_{\R^3\times S^2} |\xi-\xi_\ast|^\ga q_0(\theta) \FM (\xi_\ast)\,d\om
  d\xi_\ast\sim (1+|\xi|)^\ga,
\end{equation*}
and
\begin{eqnarray*}
% \nonumber to remove numbering (before each equation)
Ku(\xi)&=&\iint_{\R^3\times S^{2}}|\xi-\xi_\ast|^\ga q_0(\theta)\FM^{1/2}(\xi_\ast)\FM^{1/2}(\xi_\ast') u(\xi')\,d\om
  d\xi_\ast \\
  &&+\iint_{\R^3\times S^{2}}|\xi-\xi_\ast|^\ga q_0(\theta)\FM^{1/2}(\xi_\ast)\FM^{1/2}(\xi') u(\xi_\ast')\,d\om
  d\xi_\ast\nonumber \\
  &&-\iint_{\R^3\times S^{2}}|\xi-\xi_\ast|^\ga q_0(\theta)\FM^{1/2}(\xi_\ast)\FM^{1/2}(\xi) u(\xi_\ast)\,d\om
  d\xi_\ast\nonumber \\
  &=&\int_{\R^3}K(\xi,\xi_\ast)u(\xi_\ast)d\xi_\ast.\nonumber
\end{eqnarray*}
We list in the following lemma velocity weighted estimates on the collision frequency $\nu(\xi)$ and the integral operator $K$ with respect to the velocity  weight function
\begin{equation*}
    w_{\tau,\tilde{q}}(\xi)=\lag \xi\rag^{\ga \tau}e^{\tilde{q}\lag \xi\rag^2},\quad \tau\in \R,\quad 0\leq \tilde{q}\ll 1.
\end{equation*}

\begin{lemma}[\cite{SG}]
Let $-3<\gamma<0$,  $\tau\in \R$, and $0\leq \tilde{q}\ll 1$. If $|\be|>0$, then for any $\eta>0$, there is $C_\eta>0$ such that
\begin{multline*}
\int_{\R^3}w_{\tau,\tilde{q}}^2(\xi) \pa_\be (\nu u)\pa_\be u\,d\xi \geq \int_{\R^3} \nu(\xi) w_{\tau,\tilde{q}}^2(\xi)
|\pa_\be u|^2\,d\xi\\
 -\eta \sum_{|\be_1|\leq |\be|}\int_{\R^3} \nu(\xi)w_{\tau,\tilde{q}}^2(\xi)|\pa_{\be_1} u|^2\,d\xi
-C_\eta \int_{\R^3} \int_{\R^3}\chi_{|\xi|\leq 2C_\eta} \lag \xi\rag^{2\ga \tau}|u|^2\,d\xi.
\end{multline*}
If $|\be|\geq 0$, then for any $\eta>0$, there is $C_\eta>0$ such that
\begin{multline}\label{lem.nuk.2}
\left|\int_{\R^3} w_{\tau,\tilde{q}}^2(\xi)\pa_\be \left(K f\right) g\,d\xi\right|\leq \left\{\eta \sum_{|\be_1|\leq |\be|} \left(\int_{\R^3}\nu(\xi) w_{\tau,\tilde{q}}^2(\xi) |\pa_{\be_1} f|^2\,d\xi\right)^{\frac{1}{2}}\right.\\
\left.+C_\eta \left(\int_{\R^3}\chi_{|\xi|\leq 2C_\eta}\lag \xi \rag^{2\ga \tau} |f|^2\,d\xi\right)^{\frac{1}{2}}\right\}\times
\left(\int_{\R^3}\nu(\xi) w_{\tau,\tilde{q}}^2(\xi) |g|^2\,d\xi\right)^{\frac{1}{2}}.
\end{multline}
\end{lemma}

For later use, let us write down the macroscopic equations of the VPB system up to third-order moments by applying the macro-micro decomposition \eqref{mm.decom} introduced in \cite{Guo-IUMJ}.
For that, as in \cite{DS-VPB}, define moment functions $\Theta_{ij}(\cdot)$ and $\Lambda_i(\cdot)$, $1\leq i,j\leq 3$, by
\begin{equation}\label{def.moment}
    \Theta_{ij}(v)=\int_{\R^3}(\xi_i\xi_j -1)\FM^{1/2} v\,d\xi,\ \ \Lambda_i(v)=\frac{1}{10}\int_{\R^3} (|\xi|^2-5)\xi_i\FM^{1/2} v\,d\xi,
\end{equation}
for any $v=v(\xi)$. Then, one can derive from \eqref{vpb1}-\eqref{vpb2} a fluid-type system of equations
\begin{equation}\label{moment.l}
    \left\{\begin{array}{l}
       \dis \pa_t a + \na_x\cdot b =0,\\[3mm]
\dis \pa_t b+ \na_x (a + 2c) + \na_x \cdot \Theta (\{\FI-\FP\} u) -\na_x\phi=\na_x\phi a,\\[3mm]
\dis \pa_t c + \frac{1}{3}\na_x\cdot b +\frac{5}{3} \na_x\cdot  \Lambda (\{\FI-\FP\} u)=\frac{1}{3} \na_x\phi \cdot b,\\[3mm]
\dis \De_x \phi =a,
    \end{array}\right.
\end{equation}
and
\begin{equation}\label{moment.h}
    \left\{\begin{array}{l}
\dis \pa_t \Theta_{ij}(\{\FI-\FP\} u)+\pa_i b_j +\pa_j b_i -\frac{2}{3} \de_{ij} \na_x\cdot b -\frac{10}{3} \de_{ij}\na_x\cdot \Lambda (\{\FI-\FP\} u)\\[3mm]
\dis \qquad\qquad =\Theta_{ij} (r+G)-\frac{2}{3}\de_{ij}\na_x\phi \cdot b,\\[3mm]
\dis \pa_t \Lambda_i (\{\FI-\FP\} u) +\pa_i c =\Lambda_i(r+G)
    \end{array}\right.
\end{equation}
with
\begin{equation*}
    r=-\xi\cdot \na_x \{\FI-\FP\}u+\FL u,\quad G=\Ga(u,u)+\frac{1}{2} \xi\cdot \na_x\phi u-\na_x\phi \cdot \na_\xi u,
\end{equation*}
where $r$ is a linear term related only to the micro component $\{\FI-\FP\} u$ and $G$ is a quadratic nonlinear term.  Here and hereafter, for simplicity,
we used $\pa_j$ to denote $\pa_{x_j}$ for each $j=1,2,3$. The above fluid-type system \eqref{moment.l}-\eqref{moment.h} plays a key role in the analysis of the zero-order energy estimate  and the dissipation of the macroscopic component $(a,b,c)$ in the case of the whole space; see \cite{Guo-IUMJ,DY-09VPB,DS-VPB,DS-VMB,D-1VMB}.

\section{Time decay for the evolution operator}\label{sec3}

Consider the linearized homogeneous equation
\begin{equation}
\pa_t u +\xi \cdot \na_x u -\na_x \phi \cdot \xi \FM^{1/2} =\FL u, \label{lvpb1}
\end{equation}
with
\begin{equation}
  \phi=-\frac{1}{4\pi |x|}\ast_x\int_{\R^3} \FM^{1/2} u(t,x,\xi) \,d\xi\to 0 \ \text{as}\ |x|\to \infty.\label{lvpb2}
\end{equation}
Given $u_0=u_0(x,\xi)$,  $e^{t B} u_0$ is  the solution to the Cauchy problem \eqref{lvpb1}-\eqref{lvpb2} with $u|_{t=0}=u_0$.  For an integer $m$, set the index $ \si_{m}$ of the time-decay rate  by
\begin{equation*}
    \si_{m}=\frac{3}{4}+\frac{m}{2},
\end{equation*}
which corresponds to the one for the case of the usual heat kernel  in three dimensions.

The main result of this section is stated as follows.

\begin{theorem}\label{thm.lide}
Set $\mu=\mu(\xi):=\lag \xi \rag^{-\frac{\ga}{2}}$. Let $-3<\ga<0$, $\ell\geq 0$, $\al\geq 0$, $m=|\al|$,  $\ell_0>2\si_m$,  and assume
\begin{equation}\label{thm.lide.1}
    \int_{\R^3}a_0\,dx=0,\quad \int_{\R^3}(1+|x|)|a_0|\,dx<\infty,
\end{equation}
and
\begin{equation}\label{thm.lide.2}
   \left\|\mu^{\ell+\ell_0}u_0\right\|_{Z_1}+\left\|\mu^{\ell+\ell_0} \pa^\al u_0\right\|<\infty.
\end{equation}
Then, the evolution operator $e^{tB}$ satisfies
\begin{multline}\label{thm.lide.3}
\left\|\mu^{\ell}\pa^\al e^{tB}u_0\right\| +\left\|\pa^\al \na_x\De_x^{-1}\FP_0 e^{tB} u_0\right\|\\
\leq C(1+t)^{-\si_{m}} \left(\left\|\mu^{\ell+\ell_0}u_0\right\|_{Z_1}+\left\|\mu^{\ell+\ell_0} \pa^\al u_0\right\|+\left\|(1+|x|)a_0\right\|_{L^1_x}\right)
\end{multline}
for any $t\geq 0$.
\end{theorem}

\begin{proof}
The proof is divided by three steps.

\medskip

\noindent{\bf Step 1.} The Fourier transform of \eqref{lvpb1}-\eqref{lvpb2} gives
\begin{equation}\label{thm.lide.p1}
    \pa_t \hat{u}+i\xi\cdot k \hat{u}-ik\hat{\phi}\cdot \xi \FM^{1/2}=\FL \hat{u},
\end{equation}
with $-|k|^2\hat{\phi}=\hat{a}$. By taking the complex inner product of the above equation with $\hat{u}$, integrating it over $\R^3_\xi$ and using \eqref{coerc}, as in \cite{DS-VPB}, one has
\begin{equation}\label{thm.lide.p2}
    \pa_t \left(\|\hat{u}\|_{L^2_\xi}^2+\frac{|\hat{a}|^2}{|k|^2}\right)+ \kappa \left\|\nu^{1/2}\{\FI-\FP\} \hat{u}\right\|_{L^2_\xi}^2
    \leq 0.
\end{equation}
Furthermore, similar to derive \eqref{moment.l} and \eqref{moment.h} from the nonlinear VPB system,
one can also derive from \eqref{lvpb1}-\eqref{lvpb2} a fluid-type system of linear equations
\begin{eqnarray*}
% \nonumber to remove numbering (before each equation)
&& \pa_t a + \na_x\cdot b =0,\\
&& \pa_t b+ \na_x (a + 2c) + \na_x \cdot \Theta (\{\FI-\FP\} u) -\na_x\phi=0,\\
&& \pa_t c + \frac{1}{3}\na_x\cdot b +\frac{5}{3} \na_x\cdot  \Lambda (\{\FI-\FP\} u)=0,\\
&& \De_x \phi =a,
\end{eqnarray*}
and
\begin{eqnarray*}
% \nonumber to remove numbering (before each equation)
&& \pa_t \Theta_{ij}(\{\FI-\FP\} u)+\pa_i b_j +\pa_j b_i -\frac{2}{3} \de_{ij} \na_x\cdot b \\
&&\qquad\qquad\qquad\qquad-\frac{10}{3} \de_{ij}\na_x\cdot \Lambda (\{\FI-\FP\} u) =\Theta_{ij} (r),\\
&& \pa_t \Lambda_i (\{\FI-\FP\} u) +\pa_i c =\Lambda_i(r)
\end{eqnarray*}
with $r=-\xi\cdot \na_x \{\FI-\FP\}u+\FL u$, where $\Theta(\cdot)$ and $\Lambda(\cdot)$ are defined in \eqref{def.moment}. From the above fluid-type system, as in \cite{D-1VMB}, one can deduce
\begin{equation}\label{thm.lide.p3}
\pa_t \Re\, \CE^{\rm int}(\hat{u}(t,k))+ \kappa \left\{\frac{|k|^2}{1+|k|^2} \left|\widehat{(a,b,c)}\right|^2+|\hat{a}|^2\right\}
\leq C\left\|\nu^{1/2}\{\FI-\FP\} \hat{u}\right\|_{L^2_\xi}^2,
\end{equation}
where $\CE^{\rm int}(\hat{u}(t,k))$ is given by
\begin{eqnarray*}
\CE^{\rm int}(\hat{u}(t,k))&=&\frac{1}{1+|k|^2}\Bigg\{ \big(ik\hat{c}\mid \Lambda (\{\FI-\FP\} \hat{u})\big)\nonumber\\
&&\quad +\sum_{jm}\left(ik_j \hat{b}_m+i k_m \hat{b}_j-\frac{2}{3}\de_{jm} ik\cdot \hat{b}\mid \Theta_{jm}(\{\FI-\FP\} \hat{u})\right)\\
&&\quad+\kappa_1\left( ik\hat{a}\mid \hat{b} \right)\Bigg\}.\nonumber
\end{eqnarray*}
Therefore, for $0<\kappa_2\ll 1$, a suitable linear combination of \eqref{thm.lide.p2} and \eqref{thm.lide.p3} gives
\begin{multline}\label{thm.lide.p4}
\pa_t \left[\|\hat{u}\|_{L^2_\xi}^2+\frac{|\hat{a}|^2}{|k|^2}+\kappa_2\Re\, \CE^{\rm int}(\hat{u}(t,k)) \right]\\
+\kappa\left\{\left\|\nu^{1/2}\{\FI-\FP\} \hat{u}\right\|_{L^2_\xi}^2+\frac{|k|^2}{1+|k|^2} \left|\widehat{(a,b,c)}\right|^2+\left|\hat{a}\right|^2\right\}\leq 0.
\end{multline}
Here, notice that since
\begin{equation*}
   \left|\CE^{\rm int}(\hat{u}(t,k))\right|\leq C\left\{\left\|\{\FI-\FP\}\hat{u}\right\|_{L^2_\xi}^2+\left|\widehat{(a,b,c)}\right|^2\right\},
\end{equation*}
we have taken $\kappa_2>0$ small enough such that
\begin{equation}\label{thm.lide.p4-1}
    \|\hat{u}\|_{L^2_\xi}^2+\frac{|\hat{a}|^2}{|k|^2}+\kappa_2\Re\, \CE^{\rm int}(\hat{u}(t,k)) \sim \|\hat{u}\|_{L^2_\xi}^2+\frac{|\hat{a}|^2}{|k|^2}.
\end{equation}

\medskip

\noindent{\bf Step 2.} Applying $\{\FI-\FP\}$ to \eqref{thm.lide.p1}, we have
\begin{equation*}
   \pa_t \{\FI-\FP\}\hat{u}+i\xi\cdot k \{\FI-\FP\}\hat{u}=\FL \{\FI-\FP\}\hat{u} +\FP i\xi\cdot k\hat{u}-i\xi\cdot k \FP\hat{u}.
\end{equation*}
By further taking the complex inner product of the above equation with $\mu^{2\ell}(\xi)\{\FI-\FP\}\hat{u}$ and integrating it over $\R^3_\xi$, one has
\begin{multline}\label{thm.lide.p5}
\frac{1}{2}\pa_t \left\|\mu^\ell\{\FI-\FP\}\hat{u}\right\|_{L^2_\xi}+\kappa \left\|\mu^{\ell-1}\{\FI-\FP\}\hat{u}\right\|_{L^2_\xi}^2\\
\leq \Re\,\int_{\R^3}\left( K \{\FI-\FP\}\hat{u}\mid \mu^{2\ell}(\xi)\{\FI-\FP\}\hat{u}\right)\,d\xi \\
+\Re\,\int_{\R^3}\left( \FP i\xi\cdot k\hat{u}-i\xi\cdot k \FP\hat{u}\mid \mu^{2\ell}(\xi)\{\FI-\FP\}\hat{u}\right)\,d\xi.
\end{multline}
It follows from \eqref{lem.nuk.2} that for any $\eta>0$,
\begin{multline}\label{thm.lide.p5-1}
\left|\int_{\R^3}\left( K \{\FI-\FP\}\hat{u}\mid \mu^{2\ell}(\xi)\{\FI-\FP\}\hat{u}\right)\,d\xi\right|\\
\leq \left\{\eta  \left\|\nu^{1/2}\mu^{\ell}\{\FI-\FP\}\hat{u}\right\|_{L^2_\xi}
+C_\eta \left\|\chi_{|\xi|\leq 2C_\eta}\lag \xi\rag^{-\frac{\ga \ell}{2}}\{\FI-\FP\}\hat{u}\right\|\right\}\\
\times
\left\|\nu^{1/2}\mu^{\ell}\{\FI-\FP\}\hat{u}\right\|,
\end{multline}
which further by the Cauchy-Schwarz inequality implies
\begin{multline*}
\left|\int_{\R^3}\left( K \{\FI-\FP\}\hat{u}\mid \mu^{2\ell}(\xi)\{\FI-\FP\}\hat{u}\right)\,d\xi\right|\\
\leq \eta  \left\|\nu^{1/2}\mu^{\ell}\{\FI-\FP\}\hat{u}\right\|_{L^2_\xi}^2
+C_\eta \left\|\nu^{1/2}\{\FI-\FP\}\hat{u}\right\|^2.
\end{multline*}
Notice that over $|k|\leq 1$,
\begin{multline*}
\Re\,\int_{\R^3}\left( \FP i\xi\cdot k\hat{u}-i\xi\cdot k \FP\hat{u}\mid \mu^{2\ell}(\xi)\{\FI-\FP\}\hat{u}\right)\,d\xi\,\chi_{|k|\leq 1}\\
\leq C\left(\left\|\nu^{1/2}\{\FI-\FP\}\hat{u}\right\|_{L^2_\xi}^2+\frac{|k|^2}{1+|k|^2}\left|\widehat{(a,b,c)}\right|^2\right).
\end{multline*}
Plugging the above two inequalities into \eqref{thm.lide.p5} and fixing a  small enough constant $\eta>0$ give
\begin{multline}\label{thm.lide.p6}
\pa_t \left\|\mu^\ell\{\FI-\FP\}\hat{u}\right\|_{L^2_\xi}\,\chi_{|k|\leq 1}+\kappa \left\|\mu^{\ell-1}\{\FI-\FP\}\hat{u}\right\|_{L^2_\xi}^2\,\chi_{|k|\leq 1}\\
\leq C\left(\left\|\nu^{1/2}\{\FI-\FP\}\hat{u}\right\|_{L^2_\xi}^2+\frac{|k|^2}{1+|k|^2}\left|\widehat{(a,b,c)}\right|^2\right).
\end{multline}

To obtain the velocity-weighted estimate for the pointwise time-frequency variables over $|k|\geq 1$, we directly take the complex inner product of \eqref{thm.lide.p1} with $\mu^{2\ell}(\xi)\hat{u}$ and integrate it over $\R^3_\xi$ to get that
\begin{multline}\label{thm.lide.p7}
\frac{1}{2}\pa_t \left\|\mu^\ell\hat{u}\right\|_{L^2_\xi}^2+\left\|\nu^{1/2}\mu^{\ell}\{\FI-\FP\}\hat{u}\right\|_{L^2_\xi}^2\\
= -\Re\,\int_{\R^3}\left( \nu(\xi)\{\FI-\FP\}\hat{u}\mid \mu^{2\ell}(\xi)\FP \hat{u}\right)\,d\xi+\Re\,\int_{\R^3}\left( K\{\FI-\FP\} \hat{u}\mid \mu^{2\ell}(\xi)\hat{u}\right)\,d\xi\\
+\Re\,\int_{\R^3}\left( ik\hat{\phi}\cdot \xi \FM^{1/2}\mid \mu^{2\ell}(\xi)\hat{u}\right)\,d\xi.
\end{multline}
Here, Cauchy-Schwarz's inequality implies
\begin{equation*}
 -\Re\,\int_{\R^3}\left( \nu(\xi)\{\FI-\FP\}\hat{u}\mid \mu^{2\ell}(\xi)\FP \hat{u}\right)\,d\xi
 \leq C\left(\left\|\nu^{1/2}\{\FI-\FP\}\hat{u}\right\|_{L^2_\xi}^2+\left|\widehat{(a,b,c)}\right|^2\right),
\end{equation*}
and
\begin{multline*}
\Re\,\int_{\R^3}\left( ik\hat{\phi}\cdot \xi \FM^{1/2}\mid \mu^{2\ell}(\xi)\hat{u}\right)\,d\xi\\
=\Re\,\int_{\R^3}\left( ik\hat{\phi}\cdot \xi \FM^{1/2}\mid \mu^{2\ell}(\xi)[\FP\hat{u}+\{\FI-\FP\}\hat{u}]\right)\,d\xi\\
\leq C\left\{ \left\|\nu^{1/2}\{\FI-\FP\} \hat{u}\right\|_{L^2_\xi}^2+\left|\widehat{(a,b,c)}\right|^2+\frac{|\hat{a}|^2}{|k|^2}\right\},
\end{multline*}
where $|k|^2|\hat{\phi}|^2=|\hat{a}|^2/|k|^2$ due to the Poisson equation $-|k|^2\hat{\phi}=\hat{a}$ was used. From  \eqref{lem.nuk.2}, similar to the proof of \eqref{thm.lide.p5-1}, one has
\begin{multline*}
\Re\,\int_{\R^3}\left( K\{\FI-\FP\} \hat{u}\mid \mu^{2\ell}(\xi)\hat{u}\right)\,d\xi\\
=\Re\,\int_{\R^3}\left( K\{\FI-\FP\} \hat{u}\mid \mu^{2\ell}(\xi)[\FP\hat{u}+\{\FI-\FP\}\hat{u}]\right)\,d\xi\\
\leq \eta \left\|\nu^{1/2}\mu^{\ell}\{\FI-\FP\} \hat{u}\right\|_{L^2_\xi}^2
 +C_\eta\left(\left\|\nu^{1/2}\{\FI-\FP\}\hat{u}\right\|_{L^2_\xi}^2+\left|\widehat{(a,b,c)}\right|^2\right).
\end{multline*}
By plugging the above two estimates into \eqref{thm.lide.p7}, fixing a small constant $\eta>0$ and using $1\leq \frac{2|k|^2}{1+|k|^2}\chi_{|k|\geq 1}$, it follows that
\begin{multline}\label{thm.lide.p8}
\pa_t \left\|\mu^\ell\hat{u}\right\|_{L^2_\xi}^2\chi_{|k|\geq 1}+\kappa \left\|\mu^{\ell-1}\{\FI-\FP\}\hat{u}\right\|_{L^2_\xi}^2\chi_{|k|\geq 1}\\
\leq C\left(\left\|\nu^{1/2}\{\FI-\FP\}\hat{u}\right\|_{L^2_\xi}^2+\frac{|k|^2}{1+|k|^2}\left|\widehat{(a,b,c)}\right|^2\right).
\end{multline}

Now, for properly chosen constants $0<\kappa_3,\kappa_4\ll 1$, a suitable linear combination of \eqref{thm.lide.p4}, \eqref{thm.lide.p6} and \eqref{thm.lide.p8} yields that 
whenever $\ell\geq 0$,
\begin{equation}\label{ad.p1}
\pa_t E_\ell(\hat{u}) +\kappa D_\ell(\hat{u})\leq 0,
\end{equation}
holds true for any $t\geq 0$, $k\in \R^3$,
where $E_\ell(\hat{u})$ and $D_\ell(\hat{u})$ are given by
\begin{eqnarray*}
% \nonumber to remove numbering (before each equation)
   E_\ell(\hat{u})&=&\|\hat{u}\|_{L^2_\xi}^2+\frac{|\hat{a}|^2}{|k|^2}+\kappa_2\Re\, \CE^{\rm int}(\hat{u}(t,k))\\
   && +\kappa_3\left\|\mu^\ell\{\FI-\FP\}\hat{u}\right\|_{L^2_\xi}\,\chi_{|k|\leq 1}
    +\kappa_4 \left\|\mu^\ell\hat{u}\right\|_{L^2_\xi}^2\chi_{|k|\geq 1},\nonumber\\
   D_\ell(\hat{u})&=& \left\|\mu^{\ell-1}\{\FI-\FP\} \hat{u}\right\|_{L^2_\xi}^2+\frac{|k|^2}{1+|k|^2} \left|\widehat{(a,b,c)}\right|^2+\left|\hat{a}\right|^2.
\end{eqnarray*}
Due to \eqref{thm.lide.p4-1} and the fact that $\FP \hat{u}$ decays exponentially in $\xi$, it is straightforward to check that for $\ell\geq 0$,
\begin{equation*}
    E_\ell(\hat{u})\sim \left\|\mu^\ell\hat{u}\right\|_{L^2_\xi}^2+ \frac{|\hat{a}|^2}{|k|^2}.
\end{equation*}

\medskip

\noindent{\bf Step 3.} 
To the end, set
\begin{equation*}
\rho(k)=\frac{|k|^2}{1+|k|^2}.
\end{equation*}
Let $0<\eps\leq 1$   and $J> 0$  be chosen later. Multiplying \eqref{ad.p1} by $[1+\eps\rho (k)t]^J$,
\begin{equation}
\label{ad.p2}
\frac{d}{dt}\{[1+\eps\rho (k)t]^J E_{\ell}(\hat{u})\}+\kappa [1+\eps\rho (k)t]^J D_{\ell}(\hat{u})\leq J[1+\eps\rho (k)t]^{J-1} \eps \rho(k) E_{\ell}(\hat{u}).
\end{equation}
To estimate the right-hand term, we bound $E_\ell(\hat{u})$ in the way that
\begin{eqnarray*}
E_\ell(\hat{u}) &\sim& \left\|\mu^\ell\hat{u}\right\|_{L^2_\xi}^2+ \frac{|\hat{a}|^2}{|k|^2}\\
&=& \left\|\mu^\ell\hat{u}\right\|_{L^2_\xi}^2(\chi_{|k|\leq 1}+\chi_{|k|>1})+ \frac{|\hat{a}|^2}{|k|^2}\\
&\lesssim&(\left\|\mu^\ell\{\FI-\FP\}\hat{u}\right\|_{L^2_\xi}^2+\left\|\mu^\ell\FP\hat{u}\right\|_{L^2_\xi}^2)\chi_{|k|\leq 1}+\left\|\mu^\ell\hat{u}\right\|_{L^2_\xi}^2\chi_{|k|\geq 1}+ \frac{|\hat{a}|^2}{|k|^2}\\
&\lesssim&\left\|\mu^\ell\{\FI-\FP\}\hat{u}\right\|_{L^2_\xi}^2\chi_{|k|\leq 1}+\left\|\mu^\ell\hat{u}\right\|_{L^2_\xi}^2\chi_{|k|\geq 1}
+ |\widehat{(a,b,c)}|^2+ \frac{|\hat{a}|^2}{|k|^2}.
\end{eqnarray*}
That is,
\begin{eqnarray*}
&&E_\ell(\hat{u})\lesssim E^{I}_\ell(\hat{u})+E^{I\!I}(\hat{u}),\\
&&E^{I}_\ell(\hat{u})=\left\|\mu^\ell\{\FI-\FP\}\hat{u}\right\|_{L^2_\xi}^2\chi_{|k|\leq 1}+\left\|\mu^\ell\hat{u}\right\|_{L^2_\xi}^2\chi_{|k|\geq 1},\\
&&E^{I\!I}(\hat{u})= |\widehat{(a,b,c)}|^2+ \frac{|\hat{a}|^2}{|k|^2}.
\end{eqnarray*}
Therefore, for the right-hand term of \eqref{ad.p2}, one has
\begin{eqnarray*}
&\dis J[1+\eps\rho (k)t]^{J-1} \eps \rho(k) E_{\ell}(\hat{u})\\
&\dis \lesssim J[1+\eps\rho (k)t]^{J-1} \eps \rho(k)E^{I}_\ell(\hat{u})+J[1+\eps\rho (k)t]^{J-1} \eps \rho(k)E^{I\!I}(\hat{u})
:=R_1+R_2.
\end{eqnarray*}
We estimate $R_1$ and $R_2$ as follows. First, for $R_2$, 
\begin{eqnarray*}
R_2&=&J[1+\eps\rho (k)t]^{J-1} \eps \rho(k)E^{I\!I}(\hat{u})\\
&=&J[1+\eps\rho (k)t]^{J-1} \eps \rho(k)( |\widehat{(a,b,c)}|^2+ \frac{|\hat{a}|^2}{|k|^2})\\
&=&\eps J[1+\eps\rho (k)t]^{J-1} \left\{\frac{|k|^2}{1+|k|^2} \left|\widehat{(a,b,c)}\right|^2+\frac{1}{1+|k|^2} \left|\hat{a}\right|^2\right\}\\
&\leq &\eps J[1+\eps\rho (k)t]^{J}  \left\{\frac{|k|^2}{1+|k|^2} \left|\widehat{(a,b,c)}\right|^2+\frac{ \left|\hat{a}\right|^2}{|k|^2}\right\}\\
&\leq &\eps J [1+\eps\rho (k)t]^{J}D_\ell(\hat{u}).
\end{eqnarray*}
Thus, one can let
\begin{equation*}
\eps J<\frac{\kappa}{4}
\end{equation*}
where $\kappa>0$ is given in \eqref{ad.p1}, such that
\begin{equation*}
R_2\leq \frac{\kappa}{4} [1+\eps\rho (k)t]^{J}D_\ell(\hat{u}).
\end{equation*}
For $R_1$, we use the splitting
\begin{eqnarray}\label{split.tf}
% \nonumber to remove numbering (before each equation)
1= \chi_{\mu^2(\xi)\leq [1+\eps \rho(k)t]}+ \chi_{\mu^2(\xi)>  [1+\eps \rho(k)t]},
\end{eqnarray}
so that
\begin{equation*}
  E_\ell^I(\hat{u})=   E_\ell^I(\hat{u} \chi_{\mu^2(\xi)\leq [1+\eps \rho(k)t]}) +  E_\ell^I(\hat{u}\chi_{\mu^2(\xi)>  [1+\eps \rho(k)t]}):= E_\ell^{I<}(\hat{u})+ E_\ell^{I>}(\hat{u})
\end{equation*}
with
\begin{eqnarray*}
E_\ell^{I<}(\hat{u}) & = & \int_{\mu^2(\xi)\leq [1+\eps \rho(k)t]} \mu^{2\ell}(\xi) |\{\FI-\FP\}\hat{u}|^2\,d\xi\chi_{|k|\leq 1} \\
&&+ \int_{\mu^2(\xi)\leq [1+\eps \rho(k)t]} \mu^{2\ell}(\xi) |\hat{u}|^2\,d\xi\chi_{|k|\geq 1}, \\
E_\ell^{I>}(\hat{u})& = & \int_{\mu^2(\xi)\geq [1+\eps \rho(k)t]} \mu^{2\ell}(\xi) |\{\FI-\FP\}\hat{u}|^2\,d\xi\chi_{|k|\leq 1} \\
&&+ \int_{\mu^2(\xi)\geq [1+\eps \rho(k)t]} \mu^{2\ell}(\xi) |\hat{u}|^2\,d\xi\chi_{|k|\geq 1}.
\end{eqnarray*}
Thus, with this splitting,  $R_1$ is bounded by
\begin{eqnarray*}
R_1&=&J[1+\eps\rho (k)t]^{J-1} \eps \rho(k)E^{I}_\ell(\hat{u})\\
&\leq &J[1+\eps\rho (k)t]^{J-1} \eps \rho(k)\{E_\ell^{I<}(\hat{u}) +E_\ell^{I>}(\hat{u})\}\\
&\leq& \eps J [1+\eps\rho (k)t]^{J}\rho(k) E_{\ell-1}^{I<}(\hat{u})\\
&&+J[1+\eps\rho (k)t]^{J-1} \eps \rho(k) E_\ell^{I>}(\hat{u}):=R_{11}+R_{12},
\end{eqnarray*}
where we have used
$$
E^{I<}_\ell(\hat{u})\leq [1+\eps \rho(k)t] E_{\ell-1}^{I<}(\hat{u}).
$$
For $R_{11}$, notice that
\begin{eqnarray*}
\rho(k) E_{\ell-1}^{I<}(\hat{u})&\leq &\rho(k) (\left\|\mu^{\ell-1}\{\FI-\FP\}\hat{u}\right\|_{L^2_\xi}^2\chi_{|k|\leq 1}+\left\|\mu^{\ell-1}\hat{u}\right\|_{L^2_\xi}^2\chi_{|k|\geq 1})\\
&\lesssim&\left\|\mu^{\ell-1}\{\FI-\FP\}\hat{u}\right\|_{L^2_\xi}^2+\rho(k) \left|\widehat{(a,b,c)}\right|^2\\
&\leq & C_1 D_{\ell}(\hat{u})
\end{eqnarray*}
for a generic constant $C_1\geq 1$, and hence
\begin{equation*}
R_{11}\leq C_1\eps J [1+\eps\rho (k)t]^{J}D_\ell(\hat{u}).
\end{equation*}
One can make $\eps J$ further small such that 
\begin{equation*}
C_1\eps J \leq \frac{\kappa}{4},
\end{equation*}
and hence
\begin{equation*}
R_{11}\leq\frac{\kappa}{4}[1+\eps\rho (k)t]^{J}D_\ell(\hat{u}).
\end{equation*}
To estimate $R_{12}$, 
let $p>1$ be chosen later and compute
\begin{multline*}
R_{12}=J[1+\eps\rho(k)t]^{-p}\eps\rho(k)\cdot [1+\eps\rho(k)t]^{J+p-1}  E_\ell^{I>}(\hat{u})\\
  \leq  J [1+\eps \rho(k)t]^{-p}\eps \rho(k) \cdot E_{\ell+J+p-1}^{I>}(\hat{u}).
\end{multline*}
Noticing that the estimate
\begin{equation*}
  E_{\ell+J+p-1}^{I>}(\hat{u})\leq  E_{\ell+J+p-1}^I(\hat{u})\lesssim  E_{\ell+J+p-1}(\hat{u})\leq  E_{\ell+J+p-1}(\hat{u}_0)
\end{equation*}
holds true by \eqref{ad.p1} due to $\ell+J+p-1\geq 0$, one has
\begin{equation*}
\label{ }
R_{12}\leq  J [1+\eps \rho(k)t]^{-p}\eps \rho(k) E_{\ell+J+p-1}(\hat{u}_0).
\end{equation*}
Therefore, collecting all estimates above, it follows from \eqref{ad.p2} that
\begin{multline*}
\dis \frac{d}{dt}\{[1+\eps\rho (k)t]^J E_{\ell}(\hat{u})\}+\frac{\kappa}{2} [1+\eps\rho (k)t]^J D_{\ell}(\hat{u})\\
\dis \leq  J [1+\eps \rho(k)t]^{-p}\eps \rho(k) E_{\ell+J+p-1}(\hat{u}_0).
\end{multline*}
Integrating this inequality, using
\begin{equation*}
    \int_0^t [1+\eps \rho(k)s]^{-p}\eps\rho(k) ds\leq \int_0^\infty (1+\eta)^{-p}d\eta= C_p<\infty
\end{equation*}
for $p>1$, and noticing  $J+p-1>0$, one has
\begin{equation*}
   [1+\eps \rho(k)t]^J E_\ell(\hat{u})\leq  E_\ell(\hat{u}_0)+JC_pE_{\ell+J+p-1}(\hat{u}_0)\lesssim (1+JC_p)E_{\ell+J+p-1}(\hat{u}_0),
\end{equation*}
that is, whenever $\ell\geq 0$,
\begin{equation*}
   E_\ell(\hat{u})\lesssim   [1+\eps \rho(k)t]^{-J} E_{\ell+J+p-1}(\hat{u}_0),
\end{equation*}
for any $t\geq 0$, $k\in \R^3$, where the parameters $p$, $\eps$ and $J$ with
\begin{equation*}
p>1,\  0<\eps\leq 1, \ J>0,\  C_1\eps J \leq \frac{\kappa}{4}.
\end{equation*}
are still to be chosen.

Now, for any given $\ell_0>2\si_m$, we fix $J>2\si_m$, $p>1$ such that $\ell_0=J+p-1$ to have
\begin{equation*}
   E_\ell(\hat{u})\lesssim    [1+\eps \rho(k)t]^{-J} E_{\ell+\ell_0}(\hat{u}_0).
\end{equation*}
Since $J> 2\si_m$, in the completely same way as in \cite{DS-VPB} and \cite{DYZ} by considering the frequency integration over $\R^3_k=\{|k|\leq 1\}\cup\{|k|\geq 1\}$ with a little modification of the proof in \cite{Ka,Ka-BE13}, one can derive the desired time-decay property \eqref{thm.lide.3} under conditions \eqref{thm.lide.1} and \eqref{thm.lide.2}; the details are omitted for brevity. This completes the proof of Theorem \ref{thm.lide}.
\end{proof}

\begin{remark}\label{rem.lide}
The rate index $\si_m$ in \eqref{thm.lide.3} is optimal in the sense that it is the same as in the case of the Boltzmann equation. In fact, from \cite{Ukai-1974,UY-AA}, one has
$$
\left\|\pa^\al e^{t(-\xi\cdot \na_x +\FL)}u_0\right\|\leq C(1+t)^{-\si_{|\al|}} \left(\|u_0\|_{Z_1}+\|\pa^\al u\|\right),
$$
for any $t\geq 0$, where $\FL$ the linearized Boltzmann operator in the case $0\leq \ga\leq 1$. Here, different from the hard potential case, the extra velocity weight $\lag \xi\rag^{-\frac{\ga\ell_0}{2}}$ in \eqref{thm.lide.3}  on initial data $u_0$ for soft potentials $-3<\ga<0$ is needed. Moreover, as we mentioned before, the condition \eqref{thm.lide.1} is used to remove the singularity of the Poisson kernel and recover the optimal rate index $\si_m$, otherwise, only the rate $(1+t)^{-\frac{1}{4}+\frac{|\al|}{2}}$ can be obtained, see \cite{DS-VPB} and references therein.
\end{remark}

%\newpage
\section{Estimate on nonlinear terms}\label{sec4}

This section concerns some estimates on each nonlinear term in $G=\Ga(u,u)+\frac{1}{2} \xi\cdot \na_x\phi u-\na_x\phi \cdot \na_\xi u$ which will be needed in the next section. For this purpose, to simplify the notations, in the proof of all subsequent lemmas, $w_{\tau}$ is used to  denote $w_{\tau,q}(t,\xi)$. That is
\begin{equation*}
    w_{\tau}\equiv w_{\tau,q}(t,\xi)=\lag \xi \rag^{\ga \tau}e^{q(t)\lag \xi \rag^2},\quad q(t)=q+\frac{\la}{(1+t)^{\vth}},
\end{equation*}
where $0<q(t)\ll 1$ is fixed  and $\tau=|\be|-\ell\leq 0$ depending on the order of velocity derivatives may take different values in different places. Moreover, we also use in the proof
\begin{equation*}
    u_1=\FP u,\quad u_2=\{\FI-\FP\}u.
\end{equation*}

Now we turn to state the estimates on the nonlinear terms. To make the presentation clear, we divide it into several subsections. The first subsection is concerned with the
estimates on $\Ga(u,u)$.

\subsection{Estimate on $\Ga(u,u)$}

We always use the decomposition
\begin{multline}\label{ga.de}
\Ga(u,u)=\Ga(\FP u,\FP u)+\Ga(\FP u,\{\FI-\FP\}u)+\Ga(\{\FI-\FP\}u,\FP u)\\
+\Ga(\{\FI-\FP\}u,\{\FI-\FP\}u).
\end{multline}
Recall
\begin{multline*}
\Ga(f,g)=\Ga^+(f,g)-\Ga^-(f,g)\\
=\int_{\R^3}|\xi-\xi_\ast|^\ga \FM^{1/2}(\xi_\ast) \left[\int_{\mathbb{S}^2}q_0(\theta) f(\xi_\ast')g(\xi')\,d\om\right]\,d\xi_\ast\\
-g(\xi)\int_{\R^3} |\xi-\xi_\ast|^\ga \FM^{1/2}(\xi_\ast) \left[\int_{\mathbb{S}^2}q_0(\theta)\,d\om\right]\,f(\xi_\ast)\,d\xi_\ast.
\end{multline*}

Our first result in this subsection is concerned with the estimate on $\Ga(u,u)$ both without and with the weight.

\begin{lemma}\label{lem.ga1}
Let $N\geq 4$, $\ell\geq 0$, $0<q(t)\ll 1$. It holds that
\begin{multline}\label{lem.ga1.1}
\lag \Ga(u,u),u\rag \leq C\|(a,b,c)\|_{H^1}\left\|\na_x (a,b,c)\right\|\cdot \left\|\nu^{1/2}\{\FI-\FP\} u\right\|\\
+ C\left\{\left\|\na_x (a,b,c)\right\|_{H^1}+\sum_{|\al|+|\be|\leq 4} \left\|\pa^\al_\be \{\FI-\FP\} u\right\|\right\}\left\|\nu^{1/2}\{\FI-\FP\} u\right\|^2
\end{multline}
and
\begin{multline}\label{lem.ga1.2}
\left\lag \Ga(u,u),w_{-\ell,q}^2(t,\xi)\{\FI-\FP\}u\right\rag \\
\leq C\|(a,b,c)\|_{H^1}\left\|\na_x (a,b,c)\right\|\cdot \left\|\nu^{1/2}\{\FI-\FP\} u\right\|\\
+C\left\{\CE_{N,\ell,q}(t)\right\}^{1/2} \left\|\nu^{1/2}w_{-\ell,q}(t,\xi)\{\FI-\FP\} u\right\|^2.
\end{multline}

\end{lemma}

\begin{proof}
We will prove  only \eqref{lem.ga1.2}, since \eqref{lem.ga1.1} can be obtained in the similar way by noticing $\lag \Ga(u,u),u\rag=\lag \Ga(u,u),\{\FI-\FP\}u\rag$. To this end, set
\begin{equation*}
    I_1=\left\lag \Ga(u,u),w_{-\ell,q}^2(t,\xi)\{\FI-\FP\}u\right\rag,
\end{equation*}
and denote $I_{1,11}, I_{1,12},I_{1,21},I_{1,22}$ to be the terms corresponding to the decomposition \eqref{ga.de}.
Now we turn to estimate these terms term by term.

\medskip

\noindent{\it Estimate on $I_{1,11}$:} Since $\Ga(u_1,u_1)$ decays exponentially in $\xi$,
\begin{equation*}
I_{1,11}\leq C\int_{\R^3}|(a,b,c)|^2\left\|\nu^{1/2}u_2\right\|_{L^2_\xi}\,dx
\leq C\|(a,b,c)\|_{H^1}\left\|\na_x(a,b,c)\right\|\cdot \left\|\nu^{1/2}u_2\right\|,
\end{equation*}
where H\"{o}lder's inequality with $1=1/3+1/6+1/2$ and Sobolev's inequalities
\begin{equation*}
    \|(a,b,c)\|_{L^3}\leq C \|(a,b,c)\|_{H^1},\quad \|(a,b,c)\|_{L^6}\leq C\|\na_x (a,b,c)\|
\end{equation*}
have been used.

\medskip

\noindent{\it Estimate on $I_{1,12}$:} For the loss term,
\begin{multline*}
I_{1,12}^-\leq C\iint_{\R^3\times \R^3} |u_2(\xi)|\cdot |(a,b,c)|\nu(\xi)\cdot w_{-\ell}^2(\xi)|u_2(\xi)|\,dxd\xi\\
\leq C \sup_{x\in \R^3} |(a,b,c)|\cdot \left\|\nu^{1/2} w_{-\ell}(\xi)u_2\right\|^2\leq C\|\na_x(a,b,c)\|_{H^1}\left\|\nu^{1/2} w_{-\ell}(\xi)u_2\right\|^2,
\end{multline*}
where we have used  in the first inequality
\begin{equation}\label{lem.ga1.p01}
\int_{\R^3} |\xi-\xi_\ast|^\ga \FM^{q'/2}(\xi_\ast)\,d\xi_\ast\leq C \lag \xi \rag^\ga
\end{equation}
for $q'>0$ and $-3<\ga<0$. For the gain term, using H\"{o}lder's inequality, $I_{1,12}^+$ is bounded by
\begin{multline}\label{lem.ga1.p02}
 \left\{\iiiint |\xi-\xi_\ast|^\ga q_0(\theta) \FM^{1/2}(\xi_\ast) w_{-\ell}^2(\xi)|u_1(\xi_\ast')u_2(\xi')|^2\,dxd\xi d\xi_\ast d\om \right\}^{1/2}\\
 \times \left\{\iiiint |\xi-\xi_\ast|^\ga q_0(\theta) \FM^{1/2}(\xi_\ast) w_{-\ell}^2(\xi)|u_2(\xi)|^2\,dxd\xi d\xi_\ast d\om \right\}^{1/2}.
\end{multline}
From \eqref{lem.ga1.p01}, the second term in \eqref{lem.ga1.p02} is bounded by $\left\|\nu^{1/2}w_{-\ell}(\xi)u_2\right\|$.  By noticing
$\lag \xi \rag \leq \min\{\lag \xi_\ast'\rag \lag \xi'\rag,\lag \xi_\ast'\rag+\lag \xi'\rag\}$ and
\begin{multline}\label{lem.ga1.p03}
 w_{-\ell}(\xi)=\lag \xi \rag^{-\ga\ell} e^{q(t)\lag \xi \rag}\\
 \leq C \lag \xi_\ast'\rag^{-\ga \ell}\lag \xi'\rag^{-\ga \ell}e^{q(t)\lag \xi_\ast'\rag}e^{q(t)\lag \xi'\rag}
 \leq C w_{-\ell}(\xi_\ast')w_{-\ell}(\xi')
\end{multline}
for $\ell\geq 0$, the first term in \eqref{lem.ga1.p02} is bounded by
\begin{eqnarray*}
% \nonumber to remove numbering (before each equation)
&&  C \left\{\int_{x,\xi,\xi_\ast,\om} |\xi-\xi_\ast|^\ga q_0(\theta) \FM^{\frac{1}{2}}(\xi_\ast) w_{-\ell}^2(\xi)\FM^{q'}(\xi_\ast')|(a,b,c)|^2|u_2(\xi')|^2\right\}^{1/2}\nonumber\\
&& \leq  C\left\{\int_{x,\xi,\xi_\ast,\om}  |\xi-\xi_\ast|^\ga q_0(\theta) w_{-\ell}^2(\xi')\FM^{\frac{q'}{2}}(\xi_\ast')|(a,b,c)|^2|u_2(\xi')|^2 \right\}^{1/2}\\
&& =C\left\{\int_{x,\xi,\xi_\ast,\om} |\xi'-\xi_\ast'|^\ga q_0(\theta) w_{-\ell}^2(\xi')\FM^{\frac{q'}{2}}(\xi_\ast')|(a,b,c)|^2|u_2(\xi')|^2 \right\}^{1/2}\nonumber\\
&& =C\left\{\int_{x,\xi,\xi_\ast,\om} |\xi-\xi_\ast|^\ga q_0(\theta) w_{-\ell}^2(\xi)\FM^{\frac{q'}{2}}(\xi_\ast)|(a,b,c)|^2|u_2(\xi)|^2 \right\}^{1/2},\nonumber
\end{eqnarray*}
where we have used $|\FP u(\xi_\ast')|\leq C |(a,b,c)|\FM^{q'/2}(\xi_\ast')$ with $0<q'<1$ in the first equation, $\FM^{1/2}(\xi_\ast)\leq C$, the inequality \eqref{lem.ga1.p03} and $w_{-\ell}(\xi_\ast')\FM^{q'/2}(\xi_\ast')\leq C$ in the second equation, the identity $|\xi-\xi_\ast|=|\xi'-\xi_\ast'|$ in the third equation and the change of variable $(\xi,\xi)\to (\xi',\xi_\ast')$ with the unit Jacobian in the last equation. Thus, from \eqref{lem.ga1.p01}, the first term in \eqref{lem.ga1.p02} is bounded by $C \sup_{x}|(a,b,c)| \cdot\left\|\nu^{1/2}w_{-\ell}(\xi)u_2\right\|$. Therefore,
\begin{equation*}
I_{1,12}^+\leq C\|\na_x(a,b,c)\|_{H^1}\left\|\nu^{1/2} w_{-\ell}(\xi)u_2\right\|^2.
\end{equation*}

\medskip

\noindent{\it Estimate on $I_{1,21}$:} This can be done in the  completely same way as for $I_{1,12}$, so that
\begin{equation*}
I_{1,21}\leq C\left\|\na_x(a,b,c)\right\|_{H^1}\left\|\nu^{1/2} w_{-\ell}(\xi)u_2\right\|^2.
\end{equation*}

\medskip

\noindent{\it Estimate on $I_{1,22}$:}  By \cite[Lemma 3, page 305]{SG},
\begin{equation}\label{lem.ga1.p04}
   I_{1,22}\leq  C\left\{\CE_{N,\ell,q}(t)\right\}^{1/2} \left\|\nu^{1/2}w_{-\ell}\{\FI-\FP\} u\right\|^2,
\end{equation}
and for brevity we skip the proof of this term.

\medskip

Now, the desired inequality \eqref{lem.ga1.2} follows by collecting all the above estimates. This completes the proof of Lemma \ref{lem.ga1}.
\end{proof}

For the weighted estimates on $\pa^\al_\be \Ga(u,u)$, we have

\begin{lemma}\label{lem.ga2}
Let $N\geq 8$, $1\leq |\al|+|\be|\leq N$, $\ell\geq |\be|$ and $0< q(t)\ll 1$. For given $u=u(t,x,\xi)$, define $u_{\al\be}$ as $u_{\al\be}=\pa^\al u$ if $|\be|=0$ and $u_{\al\be}=\pa^\al_\be\{\FI-\FP\} u$ if $|\be|\geq 1$. Then, it holds that
\begin{multline}\label{lem.ga2.1}
\left\lag \pa^\al_\be \Ga(u,u), w_{|\be|-\ell,q}^2(t,\xi)u_{\al\be} \right\rag \\
\leq C\|(a,b,c)\|_{H^{N}}\left\|\na_x (a,b,c)\right\|_{H^{N-1}} \left\|\nu^{1/2}u_{\al\be}\right\|
+C\left\{\CE_{N,\ell,q}(t)\right\}^{1/2} \CD_{N,\ell,q}(t).
\end{multline}

\end{lemma}

\begin{proof}
Take $\al,\be $ with $1\leq |\al|+|\be|\leq N$. Write
\begin{equation*}
    \pa_\be^\al \Ga(f,g)=\sum C^{\be}_{\be_0\be_1\be_2}C^\al_{\al_1,\al_2}\Ga_0\left(\pa^{\al_1}_{\be_1}f,\pa^{\al_2}_{\be_2}g\right),
\end{equation*}
where the summation is taken over $\be_0+\be_1+\be_2=\be$ and $\al_1+\al_2=\al$, and $\Ga_0$ is given by
\begin{multline*}
\Ga_0\left(\pa^{\al_1}_{\be_1}f,\pa^{\al_2}_{\be_2}g\right)=\Ga_0^+\left(\pa^{\al_1}_{\be_1}f,\pa^{\al_2}_{\be_2}g\right)-\Ga_0^-\left(\pa^{\al_1}_{\be_1}f,\pa^{\al_2}_{\be_2}g\right)\\
=\iint_{\R^3\times \mathbb{S}^2}|\xi-\xi_\ast|^\ga q_0(\theta) \pa_{\be_0} \left[\FM^{1/2}(\xi_\ast)\right]\pa_{\be_1}^{\al_1} f(\xi_\ast')\pa_{\be_2}^{\al_2} g(\xi')\,d\xi d\om\\
-\pa_{\be_2}^{\al_2} g(\xi)\iint_{\R^3\times \mathbb{S}^2} |\xi-\xi_\ast|^\ga q_0(\theta) \pa_{\be_0}\left[\FM^{1/2}(\xi_\ast)\right]\pa_{\be_1}^{\al_1} f(\xi_\ast)\,d\xi_\ast d\om.
\end{multline*}
In what follows we set
\begin{equation*}
    I_2=\left\lag \Ga_0\left(\pa^{\al_1}_{\be_1}u,\pa^{\al_2}_{\be_2}u\right), w_{|\be|-\ell,q}^2(t,\xi)u_{\al\be} \right\rag
\end{equation*}
and denote the terms corresponding to the decomposition \eqref{ga.de} as $I_{2,11}, I_{2,12}, I_{2,21}$ and $I_{2,22}$. These terms will be estimated term by term as follows.

\medskip

\noindent{\it Estimate on $I_{2,11}$:} Notice that $\Ga_0\left(\pa^{\al_1}_{\be_1}u_1,\pa^{\al_2}_{\be_2}u_1\right)$ decays exponentially in $\xi$ since $-3<\ga<0$. Then,
\begin{multline*}
I_{2,11}\leq C\int_{\R^3} \left|\pa^{\al_1}(a,b,c)\right|\cdot \left|\pa^{\al_2}(a,b,c)\right|\cdot \left\|\nu^{1/2}u_{\al\be}\right\|_{L^2_\xi}\,dx\\
\leq C\|(a,b,c)\|_{H^{N}}\left\|\na_x (a,b,c)\right\|_{H^{N-1}}\left\|\nu^{1/2}u_{\al\be}\right\|.
\end{multline*}

\medskip

\noindent{\it Estimate on $I_{2,12}$:} For the loss term, since
\begin{multline*}
I_{2,12}^{-}\leq \int_{x,\xi,\xi_\ast,\om} |\xi-\xi_\ast|^\ga q_0(\theta) \FM^{\frac{q'}{2}}(\xi_\ast) \left\{\FM^{\frac{q'}{2}}(\xi_\ast)|\pa^{\al_1}(a,b,c)|\right\} \\
\times \left|\pa_{\be_2}^{\al_2}u_2(\xi)\right|
\cdot w_{|\be|-\ell}^2(\xi)\left|u_{\al\be}(\xi)\right| \\
\leq C \int dx\, \left|\pa^{\al_1}(a,b,c)\right| \int d\xi\, \lag \xi \rag^{\ga}  w_{|\be|-\ell}^2(\xi)\left|\pa_{\be_2}^{\al_2}u_2(\xi)u_{\al\be}(\xi)\right|,
\end{multline*}
we have for $|\al_1|\leq N/2$ that
\begin{multline*}
 I_{2,12}^{-}\leq C \sup_{x} \left|\pa^{\al_1}(a,b,c)\right| \cdot \left\|\nu^{1/2} w_{|\be|-\ell}(\xi)\pa_{\be_2}^{\al_2} u_2\right\|\\
 \times
    \left\|\nu^{1/2} w_{|\be|-\ell}(\xi)u_{\al\be} \right\|
    \leq C\left\{\CE_{N,\ell,q}(t)\right\}^{1/2} \CD_{N,\ell,q}(t),
\end{multline*}
while for the case $|\al_1|\geq N/2$, we can get that
\begin{multline*}
 I_{2,12}^{-}\leq C\int dx\, \left|\pa^{\al_1}(a,b,c)\right|\cdot\left\|\nu^{1/2} w_{|\be|-\ell}(\xi)\pa_{\be_2}^{\al_2} u_2\right\|_{L^2_\xi}
 \left\|\nu^{1/2} w_{|\be|-\ell}(\xi)u_{\al\be}\right\|_{L^2_\xi}\\
 \leq C \sup_x\left\|\nu^{1/2} w_{|\be|-\ell}(\xi)\pa_{\be_2}^{\al_2} u_2\right\|_{L^2_\xi}\cdot \left\|\pa^{\al_1}(a,b,c)\right\|\cdot \left\|\nu^{1/2} w_{|\be|-\ell}(\xi)u_{\al\be}\right\|\\
 \leq C \left\{\CE_{N,\ell,q}(t)\right\}^{1/2} \CD_{N,\ell,q}(t).
\end{multline*}

For the gain term $I_{2,12}^+$,
\begin{multline*}
I_{2,12}^{+}\leq \int_{x,\xi,\xi_\ast,\om} |\xi-\xi_\ast|^\ga q_0(\theta) \FM^{\frac{q'}{2}}(\xi_\ast) \left\{\FM^{\frac{q'}{2}}(\xi_\ast')|\pa^{\al_1}(a,b,c)|\right\} \\
\times \left|\pa_{\be_2}^{\al_2}u_2(\xi')\right|
\cdot w_{|\be|-\ell}^2(\xi)\left|u_{\al\be}(\xi)\right|.
\end{multline*}
We now turn to deduce an estimate on the right-hand side of the above inequality by considering the two cases $|\al_1|\leq N/2$ and $|\al_1|>N/2$, respectively. In the case of  $|\al_1|\leq N/2$, we have from  H\"{o}lder's inequality that

\begin{multline}\label{lem.ga2.p1}
I_{2,12}^{+}\leq \int dx\,\left|\pa^{\al_1}(a,b,c)\right|\\
\left\{\int_{\xi,\xi_\ast,\om}  |\xi-\xi_\ast|^\ga q_0(\theta) \FM^{\frac{q'}{2}}(\xi_\ast) \FM^{\frac{q'}{2}}(\xi_\ast')
w_{|\be|-\ell}^2(\xi)\left|\pa_{\be_2}^{\al_2}u_2(\xi')\right|^2\right\}^{1/2}\\
\left\{\int_{\xi,\xi_\ast,\om}  |\xi-\xi_\ast|^\ga q_0(\theta) \FM^{\frac{q'}{2}}(\xi_\ast) \FM^{\frac{q'}{2}}(\xi_\ast')
w_{|\be|-\ell}^2(\xi)\left|u_{\al\be}(\xi)\right|^2\right\}^{1/2}.
\end{multline}
As for estimating the first factor in \eqref{lem.ga1.p02}, we have
\begin{multline*}
\int_{\xi,\xi_\ast,\om}  |\xi-\xi_\ast|^\ga q_0(\theta) \FM^{\frac{q'}{2}}(\xi_\ast) \FM^{\frac{q'}{2}}(\xi_\ast')
w_{|\be|-\ell}^2(\xi)\left|\pa_{\be_2}^{\al_2}u_2(\xi')\right|^2\\
\leq C \int_{\xi,\xi_\ast,\om}  |\xi'-\xi_\ast'|^\ga q_0(\theta) \FM^{\frac{q'}{2}}(\xi_\ast) \FM^{\frac{q'}{2}}(\xi_\ast')
w_{|\be|-\ell}^2(\xi')w_{|\be|-\ell}^2(\xi_\ast')\left|\pa_{\be_2}^{\al_2}u_2(\xi')\right|^2\\
\leq C\int_{\xi,\xi_\ast,\om}  |\xi'-\xi_\ast'|^\ga q_0(\theta) \FM^{\frac{q'}{4}}(\xi_\ast')
w_{|\be|-\ell}^2(\xi')\left|\pa_{\be_2}^{\al_2}u_2(\xi')\right|^2\\
=C\int_{\xi,\xi_\ast,\om}  |\xi-\xi_\ast|^\ga q_0(\theta) \FM^{\frac{q'}{4}}(\xi_\ast)
w_{|\be|-\ell}^2(\xi)\left|\pa_{\be_2}^{\al_2}u_2(\xi)\right|^2,
\end{multline*}
which is bounded by $\left\|\nu^{1/2}w_{|\be|-\ell}(\xi)\pa_{\be_2}^{\al_2}u_2\right\|_{L^2_\xi}^2$. Moreover,
\begin{multline*}
\int_{\xi,\xi_\ast,\om}  |\xi-\xi_\ast|^\ga q_0(\theta) \FM^{\frac{q'}{2}}(\xi_\ast) \FM^{\frac{q'}{2}}(\xi_\ast')
w_{|\be|-\ell}^2(\xi)\left|u_{\al\be}(\xi)\right|^2\\
\leq C \int_{\xi,\xi_\ast}  |\xi-\xi_\ast|^\ga \FM^{\frac{q'}{2}}(\xi_\ast)
w_{|\be|-\ell}^2(\xi)\left|u_{\al\be}(\xi)\right|^2\\
\leq C \left\|\nu^{1/2}w_{|\be|-\ell}(\xi)u_{\al\be}\right\|_{L^2_\xi}^2.
\end{multline*}
Therefore, plugging the above two estimates into \eqref{lem.ga2.p1}, in the case when $|\al_1|\leq N/2$, we have
\begin{multline*}
I_{2,12}^{+}\leq C \sup_{x}\left|\pa^{\al_1}(a,b,c)\right|\cdot  \left\|\nu^{1/2}w_{|\be|-\ell}(\xi)\pa_{\be_2}^{\al_2}u_2\right\|\\
\times\left\|\nu^{1/2}w_{|\be|-\ell}(\xi)u_{\al\be}\right\|
\leq  C\left\{\CE_{N,\ell,q}(t)\right\}^{1/2} \CD_{N,\ell,q}(t).
\end{multline*}
The discussion for the case $|\al_1|\geq N/2$ is divided into the following three subcases: In $\D_1=\{|\xi_\ast|\geq \frac{1}{2}|\xi|\}$, we have
\begin{multline*}
I_{2,121}^+\leq C\int dx\, \left|\pa^{\al_1} (a,b,c)\right|\\
\int_{\xi,\xi_\ast,\om}|\xi-\xi_\ast|^\ga q_0(\theta) \FM^{\frac{q'}{2}}(\xi)\FM^{\frac{q'}{2}}(\xi_\ast) \left|\pa_{\be_2}^{\al_2} u_2(\xi') u_{\al\be}(\xi)\right|.
\end{multline*}
Here, by H\"{o}lder's inequality, we can deduce that
\begin{multline*}
\int_{\xi,\xi_\ast,\om}|\xi-\xi_\ast|^\ga q_0(\theta) \FM^{\frac{q'}{2}}(\xi)\FM^{\frac{q'}{2}}(\xi_\ast) \left|\pa_{\be_2}^{\al_2} u_2(\xi') u_{\al\be}(\xi)\right|\\
\leq \left\{\int_{\xi,\xi_\ast,\om}|\xi-\xi_\ast|^\ga q_0(\theta) \FM^{\frac{q'}{2}}(\xi)\FM^{\frac{q'}{2}}(\xi_\ast) \left|\pa_{\be_2}^{\al_2} u_2(\xi')\right|^2\right\}^{\frac{1}{2}}\\
\times \left\{\int_{\xi,\xi_\ast,\om}|\xi-\xi_\ast|^\ga q_0(\theta) \FM^{\frac{q'}{2}}(\xi)\FM^{\frac{q'}{2}}(\xi_\ast) \left|u_{\al\be}(\xi)\right|^2\right\}^{\frac{1}{2}}\\
\leq C\left\|\nu^{1/2}\pa_{\be_2}^{\al_2} u_2\right\|_{L^2_\xi}\left\|\nu^{1/2}u_{\al\be}\right\|_{L^2_\xi}.
\end{multline*}
Therefore, we have in this subcase that
\begin{multline}\label{lem.ga2.p2}
I_{2,12}^+\leq C\sup_{x}\left\|\nu^{1/2}\pa_{\be_2}^{\al_2} u_2\right\|_{L^2_\xi}\cdot \left\|\pa^{\al_1}(a,b,c)\right\|\cdot \left\|\nu^{1/2}u_{\al\be}\right\|\\
\leq  C\left\{\CE_{N,\ell,q}(t)\right\}^{1/2} \CD_{N,\ell,q}(t).
\end{multline}
Secondly, we consider the region $\D_2=\{|\xi_\ast|\leq \frac{1}{2}|\xi|,|\xi|\leq 1\}$. In this case, we can get that
\begin{multline*}
I_{2,122}^+\leq C\int dx\, \left|\pa^{\al_1} (a,b,c)\right|\\
\int_{\xi,\xi_\ast,\om}|\xi-\xi_\ast|^\ga q_0(\theta) \FM^{\frac{q'}{2}}(\xi_\ast)\FM^{\frac{q'}{2}}(\xi_\ast') \left|\pa_{\be_2}^{\al_2} u_2(\xi') u_{\al\be}(\xi)\right|.
\end{multline*}
Similarly as in \eqref{lem.ga2.p1},
\begin{multline*}
\int_{\xi,\xi_\ast,\om}|\xi-\xi_\ast|^\ga q_0(\theta) \FM^{\frac{q'}{2}}(\xi_\ast)\FM^{\frac{q'}{2}}(\xi_\ast') \left|\pa_{\be_2}^{\al_2} u_2(\xi') u_{\al\be}(\xi)\right|\\
\leq C\left\|\nu^{1/2}\pa_{\be_2}^{\al_2} u_2\right\|_{L^2_\xi}\left\|\nu^{1/2}u_{\al\be}\right\|_{L^2_\xi}.
\end{multline*}
Hence, as in \eqref{lem.ga2.p2}, the  above two estimates imply
$$
I_{2,122}^+\leq C\left\{\CE_{N,\ell,q}(t)\right\}^{1/2} \CD_{N,\ell,q}(t).
$$
At last, we consider the region $\D_3=\{|\xi_\ast|\leq \frac{1}{2}|\xi|,|\xi|\geq 1\}$. In this domain, $|\xi-\xi_\ast|\geq \frac{1}{2}|\xi|$ and hence
\begin{equation*}
    |\xi-\xi_\ast|^\ga \leq C |\xi|^\ga\leq C (1+|\xi|^2+|\xi_\ast|^2)^{\frac{\ga}{2}}.
\end{equation*}
Thus,
\begin{multline}\label{lem.ga2.p2-0}
I_{2,123}^+\leq C\int dx\, \left|\pa^{\al_1} (a,b,c)\right|\,\int_{\xi,\xi_\ast,\om}\left(1+|\xi|^2+|\xi_\ast|^2\right)^{\frac{\ga}{2}} q_0(\theta) \\
\times \FM^{\frac{q'}{2}}(\xi_\ast)\FM^{\frac{q'}{2}}(\xi_\ast') w_{|\be|-\ell}^2(\xi)\left|\pa_{\be_2}^{\al_2} u_2(\xi') u_{\al\be}(\xi)\right|.
\end{multline}
From H\"{o}lder's inequality,
\begin{multline*}
\int_{\xi,\xi_\ast,\om}(1+|\xi|^2+|\xi_\ast|^2)^{\frac{\ga}{2}} q_0(\theta) \FM^{\frac{q'}{2}}(\xi_\ast)\FM^{\frac{q'}{2}}(\xi_\ast') w_{|\be|-\ell}^2(\xi)\left|\pa_{\be_2}^{\al_2} u_2(\xi') u_{\al\be}(\xi)\right|\\
\leq \left\{\int_{\xi,\xi_\ast,\om}(1+|\xi|^2+|\xi_\ast|^2)^{\frac{\ga}{2}}  q_0(\theta) \FM^{\frac{q'}{2}}(\xi_\ast)\FM^{\frac{q'}{2}}(\xi_\ast') w_{|\be|-\ell}^2(\xi)\left|\pa_{\be_2}^{\al_2} u_2(\xi')\right|^2\right\}^{\frac{1}{2}}\\
\times \left\{\int_{\xi,\xi_\ast,\om}(1+|\xi|^2+|\xi_\ast|^2)^{\frac{\ga}{2}} q_0(\theta) \FM^{\frac{q'}{2}}(\xi_\ast)\FM^{\frac{q'}{2}}(\xi_\ast') w_{|\be|-\ell}^2(\xi)\left| u_{\al\be}(\xi)\right|^2\right\}^{\frac{1}{2}}\\
\leq C\left\|\nu^{1/2}w_{|\be|-\ell}(\xi)\pa_{\be_2}^{\al_2} u_2\right\|_{L^2_\xi}\cdot \left\|\nu^{1/2}w_{|\be|-\ell}(\xi) u_{\al\be}\right\|_{L^2_\xi},
\end{multline*}
where the computation similar to the estimate on the first factor in \eqref{lem.ga1.p02} by using $|\xi|^2+|\xi_\ast|^2=|\xi'|^2+|\xi_\ast'|^2$ instead of $|\xi-\xi_\ast|=|\xi'-\xi_\ast'|$   has been performed. Therefore,
\begin{multline*}
I_{2,123}^+\leq C\int dx\,\left|\pa^{\al_1} (a,b,c)\right|\\
 \times\left\|\nu^{1/2}w_{|\be|-\ell}(\xi)\pa_{\be_2}^{\al_2} u_2\right\|_{L^2_\xi}\cdot \left\|\nu^{1/2}w_{|\be|-\ell}(\xi) u_{\al\be}\right\|_{L^2_\xi}\\
 \leq C \sup_{x}\left\|\nu^{1/2}w_{|\be|-\ell}(\xi)\pa_{\be_2}^{\al_2} u_2\right\|_{L^2_\xi}\cdot \left\|\pa^{\al_1}(a,b,c)\right\|\cdot \left\|\nu^{1/2}w_{|\be|-\ell}(\xi) u_{\al\be}\right\|\\
 \leq  C\left\{\CE_{N,\ell,q}(t)\right\}^{1/2} \CD_{N,\ell,q}(t).
\end{multline*}

\medskip

\noindent{\it Estimate on $I_{2,21}$:} For the loss term, due to
\begin{multline*}
I_{2,21}^-\leq \int dx\, \left|\pa^{\al_2}(a,b,c)\right|\\
\cdot \int_{\xi,\xi_\ast,\om} |\xi-\xi_\ast|^\ga q_0(\theta)\FM^{\frac{q'}{2}}(\xi)\FM^{\frac{q'}{2}}(\xi_\ast) \left|\pa_{\be_1}^{\al_1}u_2(\xi_\ast) u_{\al\be}(\xi)\right|,
\end{multline*}
we have by H\"{o}lder's inequality that
\begin{equation*}
   I_{2,21}^-\leq C \sup_{x}\left|\pa^{\al_2}(a,b,c)\right|\cdot \left\|\nu^{1/2}\pa_{\be_1}^{\al_1}u_2\right\|\cdot \left\|\nu^{1/2}u_{\al\be}\right\|
\end{equation*}
for  $|\al_2|\leq N/2$, and
\begin{equation*}
     I_{2,21}^-\leq C \sup_{x}\left\|\nu^{1/2}\FM^{q'/4}\pa_{\be_1}^{\al_1}u_2\right\|_{L^2_\xi}\cdot \left\|\pa^{\al_2}(a,b,c)\right\|\cdot \left\|\nu^{1/2}u_{\al\be}\right\|
\end{equation*}
for  $|\al_2|\geq N/2$. Hence, in both cases,
$$
I_{2,21}^-\leq C\left\{\CE_{N,\ell,q}(t)\right\}^{1/2} \CD_{N,\ell,q}(t).
$$

For the gain term, by noticing that
\begin{multline*}
I_{2,21}^+\leq C\int dx\, \left|\pa^{\al_2}(a,b,c)\right|\int_{\xi,\xi_\ast,\om} |\xi-\xi_\ast|^\ga q_0(\theta)\\
\cdot \FM^{\frac{q'}{2}}(\xi_\ast)\FM^{\frac{q'}{2}}(\xi') w_{|\be|-\ell}^2(\xi)\left|\pa_{\be_1}^{\al_1}u_2(\xi_\ast') u_{\al\be}(\xi)\right|,
\end{multline*}
then, in the  same way as for $I_{2,12}^+$, when $|\al_2|\leq N/2$, we have
\begin{multline*}
I_{2,21}^+\leq C \sup_{x}\left|\pa^{\al_2}(a,b,c)\right|\cdot \left\|\nu^{1/2}w_{|\be|-\ell}(\xi)\pa_{\be_1}^{\al_1}u_2\right\|\\
 \times\left\|\nu^{1/2}w_{|\be|-\ell}(\xi) u_{\al\be}\right\|\leq C\left\{\CE_{N,\ell,q}(t)\right\}^{1/2} \CD_{N,\ell,q}(t).
\end{multline*}
While when $|\al_2|\geq N/2$,
\begin{equation*}
  I_{2,21}^+\leq C\left\{
  \begin{array}{ll}
    \dis \sup_{x}\left\|\nu^{\frac{1}{2}}\pa_{\be_1}^{\al_1}u_2\right\|_{L^2_\xi}\cdot  \left\|\pa^{\al_2}(a,b,c)\right\|\cdot \left\|\nu^{1/2}u_{\al\be}\right\| &\quad\dis \text{over}\ \D_{1}\cup \D_2,\\[3mm]
    \dis \sup_{x}\left\|\nu^{1/2}w_{|\be|-\ell}(\xi)\pa_{\be_1}^{\al_1} u_2\right\|_{L^2_\xi}&\\
     \dis \qquad\times\left\|\pa^{\al_2}(a,b,c)\right\|\cdot \left\|\nu^{1/2}w_{|\be|-\ell}(\xi) u_{\al\be}\right\|&\quad\dis \text{over}\ \D_{3}.
  \end{array}\right.
\end{equation*}
Therefore, for each $\al_2$,
\begin{equation*}
    I_{2,21}^+\leq C\left\{\CE_{N,\ell,q}(t)\right\}^{1/2} \CD_{N,\ell,q}(t).
\end{equation*}

\medskip

\noindent{\it Estimate on $I_{2,22}$:}  As for \eqref{lem.ga1.p04}, by \cite[Lemma 3, page 305]{SG},
\begin{equation*}
   I_{2,22}\leq C\left\{\CE_{N,\ell,q}(t)\right\}^{1/2} \CD_{N,\ell,q}(t),
\end{equation*}
and we also skip the proof of this term for brevity.

\medskip

Now, \eqref{lem.ga2.1} follows by collecting all the above estimates. This completes the proof of Lemma \ref{lem.ga2}.
\end{proof}

%\newpage
\subsection{Estimate on $\xi\cdot \na_x\phi u$}

This subsection concerns the  estimate on $\xi\cdot \na_x\phi u$. To this end, we first have the following result on the case of zeroth order derivative.

\begin{lemma}\label{lem.xiu0}
Assume that $u$ and $\phi$ satisfy the second equation of \eqref{moment.l}. It holds that
\begin{multline}\label{lem.xiu0.1}
\left\lag \frac{1}{2} \xi\cdot \na_x\phi u, u\right\rag \leq \frac{1}{2} \frac{d}{dt}\int_{\R^3} |b|^2 (a+2c)\,dx \\
+C \left\{\|(a,b,c)\|_{H^2}+\|\na_x\phi\|_{H^1} + \|\na_x\phi\|\cdot \|\na_x b\|\right\}\\
 \qquad\qquad\qquad\qquad\qquad\times\left\{\|\na_x (a,b,c)\|^2+\left\|\nu^{1/2}\{\FI-\FP\}u\right\|^2\right\}\\
+C\left \|\na_x^2\phi\right\|_{H^1} \left\|\lag \xi \rag^{1/2} \{\FI-\FP\} u\right\|^2.
\end{multline}

\end{lemma}

\begin{proof}
Setting  $I_3=\left\lag \frac{1}{2} \xi\cdot \na_x\phi u, u\right\rag$ and noticing  $u=\FP u +\{\FI-\FP\}u=u_1+u_2$,
\begin{equation*}
    I_3=\sum_{m=1}^3I_{3,m}=\left\lag \frac{1}{2} \xi\cdot \na_x\phi, u_1^2\right\rag
    +\left\lag \xi\cdot \na_x\phi, u_1\cdot u_2\right\rag\\
    +\left\lag \frac{1}{2} \xi\cdot \na_x\phi, u_2^2\right\rag.
\end{equation*}
First, for $ I_{3,1}$, as in \cite{DY-09VPB,D-1VMB}, using the second equation of \eqref{moment.l} to replace $\na_x\phi$ yields
\begin{multline*}
 I_{3,1} = \int_{\R^3}\na_x\phi\cdot b (a+2c)\,dx
  =\frac{1}{2}\frac{d}{dt}\int_{\R^3}|b|^2(a+2c)\,dx \\
  +\int_{\R^3} |b|^2\left\{\frac{5}{6}\na_x\cdot b+\frac{5}{3}\na_x\cdot \Lambda(u_2)-\frac{1}{3}\na_x\phi\cdot b\right\}\,dx\\
  +\int_{\R^3} b(a+2c)\cdot \{\na_x (a+2c)+ \na_x\cdot \Theta(u_2)-\na_x\phi a\}\,dx,
\end{multline*}
where $\Theta(\cdot)$ and $\Lambda(\cdot)$ are defined in \eqref{def.moment}. Further by the H\"{o}lder, Sobolev and Cauchy-Schwarz inequalities, the above equation implies
\begin{multline*}
 I_{3,1}\leq \frac{1}{2}\frac{d}{dt}\int_{\R^3}|b|^2(a+2c)\,dx \\
 +C \left\{\|(a,b,c)\|_{H^2}+\|\na_x\phi\|\cdot \|\na_x b\|\right\}\|\na_x (a,b,c)\|^2\\
 +C \left\|(a,b,c)\|_{H^2}\{\|\Lambda(u_2)\|^2+\|\Theta(u_2)\|^2\right\}.
\end{multline*}
Noticing
\begin{equation*}
    \|\Lambda(u_2)\|+\|\Theta(u_2)\|\leq C \left\|\nu^{1/2}u_2\right\|,
\end{equation*}
then $I_{3,1}$ is bounded by the right-hand term of \eqref{lem.xiu0.1}.
Next, for $ I_{3,2}$ and $ I_{3,3}$, one has
\begin{multline*}
I_{3,2}\leq \int_{\R^3}|\na_x\phi|\cdot \left\|\xi\nu^{-1/2}u_1\right\|_{L^2_\xi}\cdot \left\|\nu^{1/2}u_2\right\|_{L^2_\xi}\,dx\\
\leq C \int_{\R^3} |\na_x\phi|\cdot |(a,b,c)|\cdot \left\|\nu^{1/2}u_2\right\|_{L^2_\xi}\,dx\\
\leq C \|\na_x\phi\|_{H^1}\left\{\|\na_x (a,b,c)\|^2+\left\|\nu^{1/2}u_2\right\|^2\right\}
\end{multline*}
and
\begin{equation*}
  I_{3,3}\leq C \left\|\na_x\phi\right\|_{L^\infty_x}\left\|\lag \xi \rag^{1/2}u_2\right\|^2 \leq C\left\|\na_x^2\phi\right\|_{H^1}\left\|\lag \xi \rag^{1/2}u_2\right\|^2.
\end{equation*}
Therefore, \eqref{lem.xiu0.1} follows from all the above estimates. This completes the proof of Lemma \ref{lem.xiu0}.
\end{proof}

For the case of higher order derivatives with respect to  $x$ variable, we have

\begin{lemma}\label{lem.xiu1}
Let $N\geq 4$, $1\leq |\al|\leq N$, and $\ell\geq 0$. It holds that
\begin{multline}\label{lem.xiu1.1}
\left\lag \pa^\al \left(\frac{1}{2} \xi\cdot \na_x\phi u\right), w_{-\ell,q}^2(t,\xi) \pa^\al u \right\rag\\
\leq C\|\na_x^2\phi\|_{H^{N-1}}  \sum_{1\leq |\al|\leq N}\left\{\left\|\lag \xi \rag^{1/2} w_{-\ell,q}(t,\xi)\pa^\al\{\FI-\FP\} u\right\|^2
+\left\|\pa^\al (a,b,c)\right\|^2\right\}.
\end{multline}
\end{lemma}

\begin{proof}
Let $I_4$ be the left-hand term of \eqref{lem.xiu1.1}. Corresponding to $u=\FP u +\{\FI-\FP\}u$, set
\begin{multline}\label{lem.xiu1.p1}
I_4=\sum_{m=1}^3I_{4,m}=\left\lag \pa^\al \left(\frac{1}{2} \xi\cdot \na_x\phi u_1\right), w_{-\ell}^2(\xi) \pa^\al u \right\rag\\
+\left\lag \pa^\al \left(\frac{1}{2} \xi\cdot \na_x\phi u_2\right), w_{-\ell}^2(\xi) \pa^\al u_1 \right\rag
+\left\lag \pa^\al \left(\frac{1}{2} \xi\cdot \na_x\phi u_2\right), w_{-\ell}^2(\xi) \pa^\al u_2 \right\rag.
\end{multline}
For $I_{4,1}$, one has
\begin{multline*}
 I_{4,1}=\sum_{|\al_1|\leq |\al|}C^\al_{\al_1}\left\lag \frac{1}{2} \xi\cdot \na_x\pa^{\al-\al_1}\phi \pa^{\al_1}u_1, w_{|\al|-\ell}^2(\xi) \pa^\al u \right\rag\\
 \leq  C\sum_{|\al_1|\leq |\al|}\int_{\R^3} \left|\na_x\pa^{\al-\al_1}\phi\right|\cdot \left| \pa^{\al_1}(a,b,c)\right|\cdot \left\|\pa^\al u\right\|_{L^2_\xi}\,dx.
\end{multline*}
Here, when $|\al_1|\leq N/2$ and $|\al_1|<|\al|$, the integral term in the summation above is bounded by
\begin{equation*}
   \sup_{x}  \left| \pa^{\al_1}(a,b,c)\right|\cdot \left\|\na_x\pa^{\al-\al_1}\phi\right\|\cdot \left\|\pa^\al u\right\|,
\end{equation*}
and when $|\al_1|\geq N/2$ or $\al_1=\al$, it is bounded by
\begin{equation*}
    \sup_x\left|\na_x\pa^{\al-\al_1}\phi\right|\cdot \left\| \pa^{\al_1}(a,b,c)\right\|\cdot \left\|\pa^\al u\right\|.
\end{equation*}
Thus, by the Sobolev inequality, $ I_{4,1}$ is bounded by the right-hand term of  \eqref{lem.xiu1.1}.
Similarly, we have for $I_{4,2}$ that
\begin{multline*}
I_{4,2}\leq C\sum_{|\al_1|\leq |\al|} \int_{\R^3} \left|\na_x \pa^{\al-\al_1} \phi\right|\cdot \left\|\pa^{\al_1}u_2\right\|_{L^2_\xi}\cdot \left|\pa^\al (a,b,c)\right|\,dx\\
\leq C \sum_{\{|\al_1|\leq N/2\}\cap \{\al_1<\al\}} \sup_x \left\|\pa^{\al_1}u_2\right\|_{L^2_\xi}\cdot \left\|\na_x \pa^{\al-\al_1} \phi\right\|\cdot \left\|\pa^\al (a,b,c)\right\|\\
+ \sum_{\{|\al_1|\geq N/2\}\cup \{\al_1=\al\}} \sup_x \left|\na_x \pa^{\al-\al_1} \phi\right|\cdot \left\|\pa^{\al_1}u_2\right\|\cdot \left\|\pa^\al (a,b,c)\right\|,
\end{multline*}
which is also bounded by the right-hand term of  \eqref{lem.xiu1.1} by Sobolev's inequality.
Finally, for $I_{4,3}$, since
\begin{multline*}
I_{4,3}\leq C \sum_{|\al_1|\leq |\al|}\int_{\R^3} \left|\na_x \pa^{\al_1}\phi\right|\cdot
\left\||\xi|^{1/2}w_{-\ell}(\xi)\pa^{\al-\al_1}u_2\right\|_{L^2_\xi}\\
\times\left\||\xi|^{1/2}w_{-\ell}(\xi)\pa^{\al}u_2\right\|_{L^2_\xi}\,dx,
\end{multline*}
it can be further estimated as for $I_{4,1} $ and $I_{4,2}$.
Therefore, \eqref{lem.xiu1.1} follows from \eqref{lem.xiu1.p1} by collecting all the above estimates. This completes the proof of Lemma \ref{lem.xiu1}.
\end{proof}

For the case of higher order mixed derivatives with respect to both $\xi$ and $x$ variables, we have

\begin{lemma}\label{lem.xiu2}
Let $N\geq 4$, $1\leq |\al|+|\be|\leq N$ with $|\be|\geq 1$, and $\ell\geq |\be|$. It holds that
\begin{multline}\label{lem.xiu2.1}
\left\lag \pa^\al_\be \left(\frac{1}{2} \xi\cdot \na_x\phi \{\FI-\FP\}u\right),w_{|\be|-\ell,q}^2(t,\xi) \pa^\al_\be \{\FI-\FP\}u \right\rag \\
\leq C \|\na_x^2\phi\|_{H^{N-1}}\sum_{|\al|+|\be|\leq N } \left\|\lag \xi \rag^{1/2}w_{|\be|-\ell,q}(t,\xi)  \pa^\al_\be \{\FI-\FP\}u\right\|^2.
\end{multline}

\end{lemma}

\begin{proof}
Due to
\begin{multline*}
\left\lag \pa^\al_\be \left(\frac{1}{2} \xi\cdot \na_x\phi u_2\right),w_{|\be|-\ell}^2(\xi) \pa^\al_\be u_2 \right\rag\\
=\sum C_{\al_1\al_2}^\al C_{\be_1\be_2}^\be \left\lag \frac{1}{2} \pa_{\be_1}\xi\cdot \na_x\pa^{\al_1}\phi\pa_{\be_2}^{\al_2} u_2,  w_{|\be|-\ell}^2(\xi) \pa^\al_\be u_2 \right\rag,
\end{multline*}
where the summation is taken over $\al_1+\al_2=\al$ and $\be_1+\be_2=\be$ with $|\be_1|\leq 1$, and for simplicity we denote each integration term in the summation as $I_5$. We now prove \eqref{lem.xiu2.1} by considering the following two cases: For the case $|\al_1|\leq N/2$, we have
\begin{equation*}
I_5\leq C \sup_{x}\left|\na_x\pa^{\al_1}\phi\right|
 \cdot\left\||\xi|^{1/2}w_{|\be|-\ell}(\xi)\pa_{\be_2}^{\al_2} u_2\right\|\cdot \left\||\xi|^{1/2}w_{|\be|-\ell}(\xi)\pa_{\be}^{\al} u_2\right\|,
\end{equation*}
which is bounded by the right-hand term of \eqref{lem.xiu2.1}. For the case $|\al_1|\geq N/2$, it is easy to see that
\begin{equation*}
 I_5\leq C \sup_{x}\left\||\xi|^{1/2}w_{|\be|-\ell}(\xi)\pa_{\be_2}^{\al_2} u_2\right\|_{L^2_\xi}
\cdot\left\|\na_x\pa^{\al_1}\phi\right\|\cdot \left\||\xi|^{1/2}w_{|\be|-\ell}(\xi)\pa_{\be}^{\al} u_2\right\|,
\end{equation*}
which is also bounded by the right-hand term of \eqref{lem.xiu2.1}.
Therefore, \eqref{lem.xiu2} follows by combining both cases above. This completes the proof of Lemma \ref{lem.xiu2}.
\end{proof}

\subsection{Estimate on $\na_x\phi\cdot \na_\xi u$}

Now we turn to estimate $\na_x\phi\cdot \na_\xi u$. For the case with only derivatives with respect to  $x$ variable, we have

\begin{lemma}\label{lem.pu1}
Assume $-2\leq \ga<0$. Let $N\geq 4$, $1\leq |\al|\leq N$, and $\ell\geq 0$. It holds that
\begin{multline}\label{lem.pu1.1}
    \left\lag \pa^\al \left(\na_x\phi \cdot \na_\xi u\right), w_{-\ell,q}^2(t,\xi) \pa^\al u \right\rag\\
\leq C\left\|\na_x^2\phi\right\|_{H^{N-1}}\left\{\sum_{|\al|+|\be|\leq N,|\be|\leq 1}\left\|\lag \xi \rag w_{|\be|-\ell}\pa^\al_\be \{\FI-\FP\} u\|^2 +\|\na_x (a,b,c)\right\|_{H^{N-1}}^2\right\}.
\end{multline}
\end{lemma}

\begin{proof}
Denote the left-hand side of \eqref{lem.pu1.1} by $I_6$ and write it as
\begin{equation*}
    I_6=I_{6,1}+\sum_{\al_1+\al_2=\al,|\al_1|\geq 1}C^{\al}_{\al_1\al_2}\{I_{6,21}(\al_1)+I_{6,22}(\al_1)\}
\end{equation*}
with
\begin{eqnarray*}
% \nonumber to remove numbering (before each equation)
  I_{6,1} &=&  \left\lag  \na_x\phi \cdot \na_\xi \pa^\al u, w_{-\ell}^2(\xi) \pa^\al u \right\rag,\\
  I_{6,21}(\al_1)&=& \left\lag  \na_x\pa^{\al_1}\phi \cdot \na_\xi \pa^{\al_2} u_1, w_{-\ell}^2(\xi) \pa^\al u \right\rag,\\
   I_{6,22}(\al_1)&=& \left\lag  \na_x\pa^{\al_1}\phi \cdot \na_\xi \pa^{\al_2} u_2, w_{-\ell}^2(\xi) \pa^\al u \right\rag.
\end{eqnarray*}
We estimate term by  term as follows. To estimate $I_{6,1}$,  notice
\begin{multline*}
  \na_\xi w_{-\ell}^2(\xi)=(-2\ga\ell) \lag \xi \rag^{-2\ga\ell-1}\na_\xi \lag \xi \rag e^{2q(t)\lag \xi\rag^2}\\
   + \lag \xi\rag^{-2\ga\ell}e^{2q(t)\lag \xi\rag^2}\cdot 2q(t)\na_\xi \lag \xi\rag^2
   \leq C \lag \xi\rag^{1-2\ga\ell}e^{2q(t)\lag \xi\rag^2}= C\lag \xi \rag w_{-\ell}^2(\xi),
\end{multline*}
where $\lag \xi \rag\geq 1$ and the fact that both $q(t)=q+\la/(1+t)^{\vth}$ and $\na_\xi \lag \xi \rag$ are bounded by a constant independent of $t$ and $\xi$ have been used. Then, from integration by part,
\begin{equation*}
   I_{6,1} =\left\lag -\frac{1}{2} \na_x\phi\cdot \na_\xi \{ w_{-\ell}^2(\xi)\},|\pa^\al u|^2\right\rag \\
\leq C \sup_x\left|\na_x\phi\right|\cdot \left\|\lag \xi \rag^{1/2}w_{-\ell}(\xi)\pa^\al u\right\|^2,
\end{equation*}
which is bounded by the right-hand side of \eqref{lem.pu1.1}.
For $ I_{6,21}(\al_1)$, it is straightforward to estimate it by
\begin{multline*}
I_{6,21}(\al_1)\leq C\int_{\R^3}\left|\na_x\pa^{\al_1}\phi\right|\cdot \left|\pa^{\al_2}(a,b,c)\right|\cdot \left\|\pa^\al u\right\|_{L^2_\xi}\,dx\\
\leq C\left\|\na_x^2\phi\right\|_{H^{N-1}} \left\{\left\|\na_x (a,b,c)\right\|_{H^{N-1}}^2+\left\|\pa^\al u\right\|^2\right\}.
\end{multline*}
To estimate $ I_{6,22}(\al_1)$, notice $\lag \xi \rag^{\ga+2}\geq 1$ due to $-2\leq \ga<0$ so that
\begin{equation*}
    w_{-\ell}^2(\xi)\leq \lag \xi \rag w_{1-\ell}(\xi)\cdot \lag \xi \rag w_{-\ell}(\xi).
\end{equation*}
Thus,
\begin{multline*}
 I_{6,22}(\al_1)\leq C \int_{\R^3} \left|\na_x\pa^{\al_1}\phi\right|  \cdot\left\|\lag \xi \rag w_{1-\ell}(\xi)\na_\xi \pa^{\al_2} u_2\right\|_{L^2_\xi}\cdot \left\|\lag \xi \rag w_{-\ell}(\xi)\pa^\al u \right\|_{L^2_\xi}\,dx\\
 \leq C\left\|\na_x^2\phi\right\|_{H^{N-1}}\left\{\sum_{|\al|+|\be|\leq N,|\be|\leq 1}\left\|\lag \xi \rag w_{|\be|-\ell}\pa^\al_\be u_2\right\|^2\right. \\
 \left. +\left\|\lag \xi \rag w_{-\ell}(\xi)\pa^\al u \right\|^2\right\},
\end{multline*}
which is further bounded by the right-hand side of \eqref{lem.pu1.1}. By collecting all the above estimates, it then completes the proof of Lemma \ref{lem.pu1}.
\end{proof}

For the case of higher order mixed derivatives with respect to  the $\xi$ and $x$ variables, we have

\begin{lemma}\label{lem.pu2}
Assume $-2\leq \ga<0$. Let $1\leq |\al|+|\be|\leq N$ with $|\be|\geq 1$, and $\ell\geq |\be|$. It holds that
\begin{multline}\label{lem.pu2.1}
\left\lag \pa^\al_\be \left(\na_x\phi \cdot \na_\xi \{\FI-\FP\}u\right),w_{|\be|-\ell,q}^2(t,\xi) \pa^\al_\be \{\FI-\FP\}u \right\rag \\ \leq C\left\|\na_x^2\phi\right\|_{H^{N-1}} \sum_{|\al|+|\be|\leq N} \left\|\lag \xi \rag w_{|\be|-\ell,q}(t,\xi)\pa_\be^\al \{\FI-\FP\}u\right\|^2.
\end{multline}

\end{lemma}

\begin{proof}
Similar to the proof of Lemma \ref{lem.pu1}, one can rewrite the left-hand side of \eqref{lem.pu2.1} as
\begin{multline*}
\left\lag \na_x\phi \cdot \na_\xi  \pa^\al_\be \{\FI-\FP\}u,w_{|\be|-\ell}^2(\xi) \pa^\al_\be \{\FI-\FP\}u \right\rag\\
+\sum_{\al_1+\al_2=\al,|\al_1|\geq 1}C_{\al_1\al_2}^\al
\left\lag \na_x\pa^{\al_1}\phi \cdot \na_\xi  \pa^{\al_2}_\be \{\FI-\FP\}u,w_{|\be|-\ell}^2(\xi) \pa^\al_\be \{\FI-\FP\}u \right\rag.
\end{multline*}
Then, similar to arguments used to deal with $I_{6,1}$ and $I_{6,22}(\al_1)$ in  Lemma \ref{lem.pu1}, \eqref{lem.pu2.1} follows and the details are omitted for brevity. This completes the proof of Lemma \ref{lem.pu2}.
\end{proof}

\section{A priori estimates}\label{sec5}

We are ready to deduce the uniform-in-time {\it a priori} estimates on the solution to the VPB system. For this purpose, we define
the time-weighted sup-energy $ X_{N,\ell,q}(t)$ by \eqref{def.x},
and suppose that the Cauchy problem \eqref{vpb1}-\eqref{vpb3} of the VPB system admits a smooth solution $u(t,x,\xi)$ over $0\leq t<T$ for $0<T\leq \infty$. We will deduce some energy type estimates on the basis of the following {\it a priori} assumption
\begin{equation}\label{apa}
    \sup_{0\leq t<T}X_{N,\ell,q}(t)\leq \de,
\end{equation}
where $\de>0$ is a suitably chosen sufficiently small positive constant.

\subsection{Dissipation of $(a,b,c)$}

Recall the fluid-type system \eqref{moment.l}-\eqref{moment.h} derived from the VPB system \eqref{vpb1}-\eqref{vpb2} and also \eqref{def.moment} for the definition of the high-order moment functions $\Theta(\cdot)$ and $\Lambda(\cdot)$. In this subsection, we are concerned with the dissipation of the macroscopic component $(a,b,c)$.

\begin{lemma}
There is a temporal interactive functional $\CE_{N}^{\rm int}(t)$ such that
\begin{multline}
\label{lem.ma.1}
   |\CE_{N}^{\rm int}(t)|
   \leq C \left\{\|a\|^2+\sum_{|\al|\leq N-1} \left(\left\|\pa^\al \{\FI-\FP\} u(t)\right\|^2
   +\left\|\pa^\al \na_x(a,b,c)\right\|^2\right)\right\}
\end{multline}
and
\begin{multline}\label{lem.ma.2}
% \nonumber to remove numbering (before each equation)
\frac{d}{dt} \CE_{N}^{\rm int}(t) +\ka \left\{\|a\|^2+\sum_{|\al|\leq N-1}\left\|\pa^\al \na_x (a,b,c)\right\|^2\right\}\\
\leq C\sum_{|\al|\leq N} \left\|\nu^{1/2}\pa^\al \{\FI-\FP\} u\right\|^2\\
+C\CE_{N,\ell}(t)\left\{\sum_{|\al|\leq N}\left\|\nu^{1/2}\pa^\al \{\FI-\FP\}u\right\|^2+ \sum_{1\leq |\al|\leq N}\left\|\pa_x^\al (a,b,c)\right\|^2\right\}
\end{multline}
hold for any $0\leq t<T$, where the simplified notion $\CE_{N,\ell}(t)=\CE_{N,\ell,q}(t)$ with $q=0$ has been used.
\end{lemma}

\begin{proof}
Basing on the analysis of the macro fluid-type system  \eqref{moment.l}-\eqref{moment.h}, the desired estimates follow by repeating the arguments used in the proof of \cite[Theorem 5.2]{DS-VPB} for the hard-sphere case and hence details are omitted. Here, we only point out the representation of $ \CE_N^{\rm int}(t)$ as
\begin{multline*}
  \CE^{{\rm int}}_{N}(t) =\sum_{|\al|\leq N-1}\int_{\R^3}\na_x\pa^\al c\cdot \Lambda(\pa^\al\{\FI-\FP\}u)\,dx \\
  +\sum_{|\al|\leq N-1}\sum_{ij=1}^3\int_{\R^3}
\left(\pa_i\pa^\al b_j+\pa_j\pa^\al b_i-\frac{2}{3}\de_{ij}\na_x\cdot \pa^\al b\right)
\Theta_{ij}(\pa^\al\{\FI-\FP\}u)\,dx\\
-\kappa\sum_{|\al|\leq N-1}\int_{\R^3}\pa^\al a \pa^\al\na_x\cdot b \,dx
\end{multline*}
for a constant $\kappa>0$ small enough.
\end{proof}

\subsection{Construction of $\CE_{N,\ell,q}(t)$}

We are ready to prove the energy inequality \eqref{est.e}. In this subsection we consider the proof \eqref{est.e} in the following

\begin{lemma}\label{lem.ee}
Assume $-2\leq \ga<0$. Let $N\geq 8$, $\ell\geq N$ and $q(t)=q+\frac{\la}{(1+t)^{\vth}} > 0$ with $0\leq q\ll 1$, $0<\la \ll 1$ and $0<\vth\leq 1/4$. Suppose that the {\it a priori} assumption \eqref{apa} holds for $\de>0$ small enough. Then, there is $\CE_{N,\ell,q}(t)$ satisfying \eqref{def.e} such that
\begin{equation}\label{lem.ee.1}
    \frac{d}{dt} \CE_{N,\ell,q}(t)+\kappa \CD_{N,\ell,q}(t)\leq 0
\end{equation}
holds for any $0\leq t<T$, where $\CD_{N,\ell,q}(t)$ is given by \eqref{def.ed}.
\end{lemma}

\begin{proof}
It is divided by three steps as follows. Recall the VPB system \eqref{vpb1}-\eqref{vpb2}.

\medskip

\noindent{\bf Step 1.} Energy estimates without any weight: First we consider the zero$-$th order energy type estimates. Multiplying \eqref{vpb1} by $u$ and integrating it over $\R^3\times \R^3$ gives
\begin{equation*}
    \frac{1}{2}\frac{d}{dt}\left\{\|u\|^2+\|\na_x\phi\|^2\right\}-\lag \FL u, u\rag=\lag \Ga(u,u), u\rag+ \left\lag \frac{1}{2}\xi \cdot \na_x\phi u,u\right\rag.
\end{equation*}
By applying \eqref{coerc}, \eqref{lem.ga1.1} and \eqref{lem.xiu0.1} to estimate three inner product terms above and then using \eqref{apa}, one has
\begin{multline}\label{lem.ee.p1}
 \frac{1}{2}\frac{d}{dt}\left\{\|u\|^2+\|\na_x\phi\|^2-\int_{\R^3}|b|^2(a+2c)\,dx\right\} +\kappa \left\|\nu^{1/2}\{\FI-\FP\}u\right\|^2\\
 \leq C\de \|\na_x (a,b,c)\|^2+ \frac{C\de}{(1+t)^{1+\vth}}\left\|\lag \xi \rag^{1/2}\{\FI-\FP\}u\right\|^2
 \leq C\de \CD_{N,\ell,q}(t).
\end{multline}
Next we consider the energy type estimates containing only $x$ derivatives. Applying $\pa_x^\al$ with $1\leq |\al|\leq N$ to \eqref{vpb1}, multiplying it by $\pa_x^\al u$ and then integrating it over $\R^3\times \R^3$ gives
\begin{multline*}
\frac{1}{2}\frac{d}{dt}\left\{\|\pa^\al u\|^2+\|\pa^\al\na_x\phi\|^2\right\}-\left\lag \FL \pa^\al u,\pa^\al u\right\rag \\
=\left\lag \pa^\al \Ga(u,u),\pa^\al u\right\rag+ \left\lag \pa^\al \left(\frac{1}{2}\xi\cdot \na_x\phi u\right),\pa^\al u\right\rag+ \left\lag -\pa^\al (\na_x\phi\cdot \na_\xi u),\pa^\al u\right\rag.
\end{multline*}
After using \eqref{coerc}, \eqref{lem.ga2.1}, \eqref{lem.xiu1.1}, \eqref{lem.pu1.1} and then taking summation over $1\leq |\al|\leq N$, one has
\begin{multline*}
\frac{1}{2}\frac{d}{dt}\sum_{1\leq |\al|\leq N}\left\{\left\|\pa^\al u\right\|^2+\left\|\pa^\al \na_x\phi\right\|^2\right\}+\kappa \sum_{1\leq |\al|\leq N}\left\|\nu^{1/2}\pa^\al \{\FI-\FP\}u\right\|^2\\
\leq  C\|(a,b,c)\|_{H^{N}}\left\|\na_x (a,b,c)\right\|_{H^{N-1}} \sum_{1\leq |\al|\leq N}\left\|\nu^{1/2}\pa^\al u\right\|
+C\left\{\CE_{N,\ell,q}^{\rm h}(t)\right\}^{1/2} \CD_{N,\ell,q}(t)\\
+C\left\|\na_x^2\phi\right\|_{H^{N-1}}\left\{\sum_{|\al|+|\be|\leq N,|\be|\leq 1}\left\|\lag \xi \rag w_{|\be|-\ell,q}(t,\xi)\pa^\al_\be\{\FI-\FP\} u\right\|^2\right.\\
+\left\|\na_x (a,b,c)\right\|_{H^{N-1}}^2\Bigg\},
\end{multline*}
which under the {\it a priori} assumption \eqref{apa} implies
\begin{multline}\label{lem.ee.p2}
\frac{1}{2}\frac{d}{dt}\sum_{1\leq |\al|\leq N}\left\{\left\|\pa^\al u\right\|^2+\left\|\pa^\al \na_x\phi\right\|^2\right\}+\kappa \sum_{1\leq |\al|\leq N}\left\|\nu^{1/2}\pa^\al \{\FI-\FP\}u\right\|^2\\
\leq C\de \CD_{N,\ell,q}(t).
\end{multline}

\medskip

\noindent{\bf Step 2.}  Energy estimates with the weight function $w_{|\be|-\ell,q}(t,\xi)$:

\medskip

\noindent{\it Step 2.1.} By applying $\{\FI-\FP\}$ to the equation \eqref{vpb1},  the time evolution of $\{\FI-\FP\}u$ satisfies
\begin{multline}\label{eq.mi}
\pa_t \{\FI-\FP\}u+\xi \cdot \na_x \{\FI-\FP\}u + \na_x\phi \cdot \na_\xi \{\FI-\FP\}u
+\nu(\xi)\{\FI-\FP\}u\\
=K\{\FI-\FP\}u +\Ga(u,u)+\frac{1}{2} \xi \cdot \na_x \phi \{\FI-\FP\}u +\comml \FP, \CT_{\phi}\commr u,
\end{multline}
where $\comml A, B\commr=AB-BA$ denotes the commutator of two operators $A,B$, and $\CT_\phi$ is given by
\begin{equation*}
    \CT_\phi=\xi \cdot \na_x  +\na_x\phi \cdot \na_\xi  -\frac{1}{2} \xi \cdot \na_x\phi.
\end{equation*}
By multiplying \eqref{eq.mi} by $w_{-\ell,q}^2(t,\xi)\{\FI-\FP\} u$ and taking integration over $\R^3\times \R^3$, one has
\begin{multline}\label{lem.ee.p3}
\frac{1}{2}\frac{d}{dt}\| w_{-\ell,q}(t,\xi) \{\FI-\FP\} u\|^2 +\left\lag \nu(\xi),w_{-\ell,q}^2(t,\xi)|\{\FI-\FP\} u|^2\right\rag\\
+\left\lag - \frac{1}{2}\frac{d}{dt}w_{-\ell,q}^2(t,\xi),  |\{\FI-\FP\} u|^2\right\rag\\
=\left\lag K\{\FI-\FP\} u,w_{-\ell,q}^2(t,\xi)\{\FI-\FP\} u\right\rag+\left\lag \Ga(u,u),w_{-\ell,q}^2(t,\xi)\{\FI-\FP\} u\right\rag\\
+\left\lag -\na_x\phi \cdot \na_\xi \{\FI-\FP\} u,w_{-\ell,q}^2(t,\xi)\{\FI-\FP\} u\right\rag\\
+\left\lag -\frac{1}{2}\xi\cdot \na_x\phi,w_{-\ell,q}^2(t,\xi)|\{\FI-\FP\} u|^2\right\rag
+\left\lag \comml  \FP, \CT_\phi\commr u,w_{-\ell,q}^2(t,\xi)\{\FI-\FP\} u\right\rag.
\end{multline}
For the left-hand third term of \eqref{lem.ee.p3}, notice
\begin{equation*}
- \frac{1}{2}\frac{d}{dt}w_{-\ell,q}^2(t,\xi)=\frac{\la\vth}{(1+t)^{1+\vth}} \lag \xi\rag^2 w_{-\ell,q}^2(t,\xi).
\end{equation*}
From \eqref{lem.nuk.2}, H\"{o}lder's inequality in $x$ integration and Cauchy-Schwarz's inequality, the right-hand first term of \eqref{lem.ee.p3} is bounded by
\begin{multline*}
\eta\left\|\nu^{1/2}w_{-\ell,q}(t,\xi) \{\FI-\FP\} u\right\|^2\\
+C_\eta \left\|\chi_{|\xi|\leq 2C_\eta}\lag \xi\rag^{-\ga \ell}\{\FI-\FP\} u \right\| \cdot \left\|\nu^{1/2}w_{-\ell,q}(t,\xi) \{\FI-\FP\} u\right\|\\
\leq 2\eta\left\|\nu^{1/2}w_{-\ell,q}(t,\xi) \{\FI-\FP\} u\right\|^2+C_\eta\left\|\nu^{1/2}\{\FI-\FP\} u\right\|^2,
\end{multline*}
for an arbitrary constant $\eta>0$.
The right-hand second, third and fourth terms of \eqref{lem.ee.p3} are bounded by $C\de \CD_{N,\ell,q}(t)$, where we have used \eqref{lem.ga1.2} for the second term, velocity integration by part for the third term and also the {\it a priori} assumption \eqref{apa}. For the right-hand last term of \eqref{lem.ee.p3}, it is straightforward to estimate it by
\begin{multline*}
\left\lag \comml  \FP, \CT_\phi\commr u,w_{-\ell,q}^2(t,\xi)\{\FI-\FP\} u\right\rag\\
\leq \eta \left\|\nu^{1/2}\{\FI-\FP\} u\right\|^2+C_\eta\left\{\left\|\nu^{1/2}\na_x\{\FI-\FP\}u\right\|^2+\left\|\na_x (a,b,c)\right\|^2\right\}\\
+C_\eta\left\|\na_x\phi\right\|_{H^2}^2\left\{\left\|\nu^{1/2}\{\FI-\FP\} u\right\|^2+\left\|\na_x(a,b,c)\right\|^2\right\},
\end{multline*}
for $\eta>0$.
Plugging the above estimates into \eqref{lem.ee.p3} and fixing a properly small constant $\eta>0$ yield
\begin{multline}\label{lem.ee.p4}
\frac{1}{2}\frac{d}{dt}\left\| w_{-\ell,q}(t,\xi) \{\FI-\FP\} u\right\|^2
+\kappa \left\|\nu^{1/2}w_{-\ell,q}(t,\xi) \{\FI-\FP\} u\right\|^2\\
+\frac{\kappa}{(1+t)^{1+\vth}}\left\|\lag \xi \rag w_{-\ell,q}(t,\xi) \{\FI-\FP\} u\right\|^2
\leq C\de \CD_{N,\ell,q}(t)\\
+C\left\{\left\|\nu^{1/2}\{\FI-\FP\} u\right\|^2+\left\|\nu^{1/2}\na_x\{\FI-\FP\}u\right\|^2+\left\|\na_x (a,b,c)\right\|^2\right\}.
\end{multline}

\medskip

\noindent{\it Step 2.2.}
For the weighted estimate on the terms containing only $x$ derivatives, we directly use \eqref{vpb1}. In fact, take $1\leq |\al|\leq N$, and by applying $\pa_x^\al$ to \eqref{vpb1} with
\begin{equation*}
    \FL u=\FL\{\FI-\FP\}u=-\nu\{\FI-\FP\}u+K\{\FI-\FP\}u,
\end{equation*}
multiplying it by $w_{-\ell,q}^2(t,\xi)\pa_x^\al u$ and integrating over $\R^3\times \R^3$, one has
\begin{multline}\label{lem.ee.p5}
\frac{1}{2}\frac{d}{dt}\left\|w_{-\ell,q}(t,\xi)\pa^\al u\right\|^2+\left\lag \nu(\xi)\pa^\al\{\FI-\FP\}u,w_{-\ell,q}^2(t,\xi)\pa^\al u\right\rag\\
+\left\lag -\frac{1}{2}\frac{d}{dt}w_{-\ell,q}^2(t,\xi),|\pa^\al u|^2\right\rag\\
=\left\lag K\pa^\al\{\FI-\FP\}u,w_{-\ell,q}^2(t,\xi)\pa^\al u\right\rag+\left\lag \pa^\al \na_x\phi \cdot \xi \FM^{1/2}, w_{-\ell,q}^2(t,\xi) \pa^\al u \right\rag\\
+\left\lag \pa^\al \Ga(u,u) -\pa^\al (\na_x\phi\cdot \na_\xi u)+\pa^\al \left(\frac{1}{2}\xi\cdot \na_x\phi u\right),w_{-\ell,q}^2(t,\xi)\pa^\al u\right\rag.
\end{multline}
For the left-hand terms of \eqref{lem.ee.p5}, one has
\begin{multline*}
\left\lag \nu(\xi)\pa^\al\{\FI-\FP\}u,w_{-\ell,q}^2(t,\xi)\pa^\al u\right\rag\\
=\left\|\nu^{1/2}w_{-\ell,q}(t,\xi)\pa^\al\{\FI-\FP\}u\right\|^2
+\left\lag \nu(\xi)\pa^\al\{\FI-\FP\}u,w_{-\ell,q}^2(t,\xi)\pa^\al\FP u\right\rag\\
\geq \frac{1}{2}\left\|\nu^{1/2}w_{-\ell,q}(t,\xi)\pa^\al\{\FI-\FP\}u\right\|^2-C \left\|\pa^\al (a,b,c)\right\|^2
\end{multline*}
and
\begin{multline*}
\left\lag -\frac{1}{2}\frac{d}{dt}w_{-\ell,q}^2(t,\xi),|\pa^\al u|^2\right\rag
=\frac{\la\vth}{(1+t)^{1+\vth}}\left\|\lag \xi \rag w_{-\ell,q}(t,\xi)\pa^\al u\right\|^2\\
\geq \frac{\la\vth}{2(1+t)^{1+\vth}}\left\|\lag \xi \rag w_{-\ell,q}(t,\xi)\pa^\al \{\FI-\FP\}u\right\|^2
-C \left\|\pa^\al (a,b,c)\right\|^2.
\end{multline*}
For the right-hand first term of \eqref{lem.ee.p5}, one has
\begin{multline*}
\left\lag K\pa^\al\{\FI-\FP\}u,w_{-\ell,q}^2(t,\xi)\pa^\al u\right\rag\\
\leq \left\{\eta\left\|\nu^{1/2}w_{-\ell,q}(t,\xi) \pa^\al \{\FI-\FP\} u\right\|
+C_\eta \left\|\chi_{|\xi|\leq 2C_\eta}\lag \xi\rag^{\ga (-\ell)}\pa^\al\{\FI-\FP\} u \right\|\right\} \\
\times \left\|\nu^{1/2}w_{-\ell,q}(t,\xi)\pa^\al u\right\|\\
\leq C\eta\left\|\nu^{1/2}w_{-\ell,q}(t,\xi) \pa^\al \{\FI-\FP\} u\right\|^2\\
+C_\eta\left\{\left\|\nu^{1/2}\pa^\al \{\FI-\FP\} u\right\|^2 +\left\|\pa^\al (a,b,c)\right\|^2\right\},
\end{multline*}
where $\eta>0$. For the right-hand second term of \eqref{lem.ee.p5}, we have that for $\eta>0$,
\begin{multline*}
\left\lag \pa^\al \na_x\phi \cdot \xi \FM^{1/2}, w_{-\ell,q}^2(t,\xi) \pa^\al u \right\rag
\leq \eta\left\|\nu^{1/2}\pa^\al u\right\|^2+ C_\eta\left\|\pa^\al \na_x\phi\right\|^2\\
\leq C\eta\left\|\nu^{1/2}\pa^\al \{\FI-\FP\}u\right\|^2+ C_\eta \left\{\|a\|^2+\left\|\na_x (a,b,c)\right\|_{H^{N-1}}^2\right\}.
\end{multline*}
For the right-hand third term of \eqref{lem.ee.p5}, from Lemma \ref{lem.ga2}, Lemma \ref{lem.xiu1} and Lemma \ref{lem.pu1}, it is bounded by
\begin{multline}\label{def.error}
  C\de \CD_{N,\ell,q}(t)
  +\frac{C\de}{(1+t)^{1+\vth}} \left\{\sum_{|\al|+|\be|\leq N}\left\|\lag \xi \rag w_{|\be|-\ell,q}(t,\xi)\{\FI-\FP\}u\right\|^2\right.\\
+\left\|\na_x (a,b,c)\right\|_{H^{N-1}}^2\Bigg\}
  \leq  C\de \CD_{N,\ell,q}(t).
\end{multline}
Putting the above estimates into \eqref{lem.ee.p5}, taking summation over $1\leq |\al|\leq N$ and fixing a small constant $\eta>0$ give
\begin{multline}\label{lem.ee.p6}
\frac{1}{2}\frac{d}{dt}\sum_{1\leq |\al|\leq N}\left\|w_{-\ell,q}(t,\xi)\pa^\al u\right\|^2\\
+\kappa\sum_{1\leq |\al|\leq N}\left\|\nu^{1/2}w_{-\ell,q}(t,\xi)\pa^\al\{\FI-\FP\}u\right\|^2\\
+\frac{\kappa}{(1+t)^{1+\vth}}\sum_{1\leq |\al|\leq N}\left\|\lag \xi \rag^{1/2} w_{-\ell,q}(t,\xi)\pa^\al u\right\|^2\\
\leq  C\de \CD_{N,\ell,q}(t)
+C\left\{\sum_{1\leq |\al|\leq N}\left\|\nu^{1/2}\pa^\al \{\FI-\FP\} u\right\|^2 +\left\|\na_x (a,b,c)\right\|_{H^{N-1}}^2+\|a\|^2\right\}.
\end{multline}

\medskip

\noindent{\it Step 2.3.} For the weighted estimate on the mixed $x$-$\xi$ derivatives, we  use the equation \eqref{eq.mi} of $\{\FI-\FP\} u$.
Let $1\leq m\leq N$. By applying $\pa_\be^\al$ with $|\be|=m$ and $|\al|+|\be|\leq N$ to \eqref{eq.mi}, multiplying it by $w_{|\be|-\ell,q}^2(t,\xi)\pa_\be^\al \{\FI-\FP\} u$ and  integrating over $\R^3\times \R^3$, one has
\begin{multline}\label{lem.ee.p7}
\frac{1}{2}\frac{d}{dt}\left\|w_{|\be|-\ell,q}(t,\xi)\pa_\be^\al \{\FI-\FP\}u\right\|^2\\
+\left\lag \pa_{\be}\{\nu(\xi)\pa^\al \{\FI-\FP\}u\},w_{|\be|-\ell,q}^2(t,\xi)\pa_\be^\al \{\FI-\FP\}u\right\rag\\
+\left\lag -\frac{1}{2}\frac{d}{dt}w_{|\be|-\ell,q}^2(t,\xi), |\pa_\be^\al \{\FI-\FP\}u|^2\right\rag\\
=\left\lag \pa_\be K \pa^\al\{\FI-\FP\}u, w_{|\be|-\ell,q}^2(t,\xi)\pa_\be\pa^\al \{\FI-\FP\} u\right\rag\\
+\Bigg\lag\pa_\be^\al \left(\Ga(u,u)-\na_x\phi\cdot \na_\xi \{\FI-\FP\}u+\frac{1}{2}\xi\cdot \na_x\phi \{\FI-\FP\}u\right), \\
w_{|\be|-\ell,q}^2(t,\xi)\pa_\be^\al \{\FI-\FP\}u \Bigg\rag\\
+\left\lag -\coml \pa_\be^\al, \xi\cdot \na_x\comr \{\FI-\FP\}u, w_{|\be|-\ell,q}^2(t,\xi)\pa_\be^\al \{\FI-\FP\} u\right\rag\\
+ \left\lag \pa_\be^\al \coml\FP, \CT_\phi\comr u, w_{|\be|-\ell,q}^2(t,\xi)\pa_\be^\al \{\FI-\FP\} u\right\rag.
\end{multline}
For the left-hand second term of \eqref{lem.ee.p7}, one has
\begin{multline*}
\left\lag \pa_{\be}\{\nu(\xi)\pa^\al \{\FI-\FP\}u\},w_{|\be|-\ell,q}^2(t,\xi)\pa_\be\pa^\al \{\FI-\FP\}u\right\rag\\
\geq \left\|\nu^{1/2}w_{|\be|-\ell,q}(t,\xi)\pa_\be^\al \{\FI-\FP\}u\right\|^2\\
-\eta\sum_{|\be_1|\leq |\be|} \left\|\nu^{1/2}w_{|\be|-\ell,q}(t,\xi)\pa_{\be_1}^\al \{\FI-\FP\}u\right\|^2\\
-C_\eta \left\|\chi_{|\xi|\leq 2C_\eta} \lag \xi\rag^{\ga (|\be|-\ell)} \pa^\al \{\FI-\FP\}u\right\|^2,
\end{multline*}
where $\eta>0$ is arbitrary, and it is noticed that
\begin{equation*}
    \left\|\chi_{|\xi|\leq 2C_\eta} \lag \xi\rag^{\ga (|\be|-\ell)} \pa^\al \{\FI-\FP\}u\right\|\leq C \left\|\nu^{1/2} \pa^\al \{\FI-\FP\}u\right\|.
\end{equation*}
For the right-hand first term of \eqref{lem.ee.p7}, we have that for $\eta>0$,
\begin{multline*}
\left\lag \pa_\be K \pa^\al\{\FI-\FP\}u, w_{|\be|-\ell,q}^2(t,\xi)\pa_\be^\al \{\FI-\FP\} u\right\rag\\
\leq \left\{\eta \sum_{|\be_1|\leq |\be|} \left\|\nu^{1/2}w_{|\be|-\ell,q}(t,\xi)\pa_{\be_1}^\al\{\FI-\FP\}u\right\|\right.\\
+C\eta\left\|\chi_{|\xi|\leq 2C_\eta}\lag \xi \rag^{\ga (|\be|-\ell)} \pa^\al\{\FI-\FP\}u\Bigg\|\right\}
\cdot
\left\|\nu^{1/2}w_{|\be|-\ell,q}(t,\xi)\pa_{\be}^\al\{\FI-\FP\}u\right\|\\
\leq C\eta \sum_{|\be_1|\leq |\be|} \left\|\nu^{1/2}w_{|\be|-\ell,q}(t,\xi)\pa_{\be_1}^\al\{\FI-\FP\}u\right\|^2\\
+C_\eta\left\|\nu^{1/2}\pa^\al\{\FI-\FP\}u\right\|^2.
\end{multline*}
As in \eqref{def.error}, from Lemma \ref{lem.ga2}, Lemma \ref{lem.xiu2} and Lemma \ref{lem.pu2},  the right-hand second term of \eqref{lem.ee.p7} is bounded by  $C\de \CD_{N,\ell,q}(t)$.

For the right-hand third term of \eqref{lem.ee.p7},
\begin{multline*}
\left\lag -\coml \pa_\be^\al, \xi\cdot \na_x\comr \{\FI-\FP\}u, w_{|\be|-\ell,q}^2(t,\xi)\pa_\be^\al \{\FI-\FP\} u\right\rag\\
=\sum_{\be_1+\be_2=\be,|\be_1|=1}C_{\be_1\be_2}^{\be}\left\lag -\pa_{\be_1}\xi\cdot \na_x\pa_{\be_2}^\al\{\FI-\FP\}u, w_{|\be|-\ell,q}^2(t,\xi)\pa_\be^\al \{\FI-\FP\} u\right\rag\\
\leq  \eta \left\|\nu^{1/2} w_{|\be|-\ell,q}(t,\xi)\pa_\be^\al \{\FI-\FP\} u\right\|^2\\
+ C_\eta \sum_{|\al|+|\be|\leq N,|\be_1|=|\be|-1} \left\|\nu^{1/2} w_{|\be_2|-\ell,q}(t,\xi)\pa_{\be_2}^\al \{\FI-\FP\} u\right\|^2,
\end{multline*}
where $\eta>0$ is arbitrary, and we have used
\begin{multline*}
w_{|\be|-\ell, q}^2(t,\xi)=w_{|\be|+\frac{1}{2}-\ell, q}(t,\xi)w_{|\be_2|+\frac{1}{2}-\ell, q}(t,\xi)\\
\leq C \nu^{1/2}(\xi) w_{|\be|-\ell, q}(t,\xi)\cdot \nu^{1/2}(\xi)w_{|\be_2|-\ell, q}(t,\xi).
\end{multline*}
The right-hand last term of \eqref{lem.ee.p7} is estimated as

\begin{multline*}
 \left\lag \pa_\be^\al \coml\FP, \CT_\phi\comr u, w_{|\be|-\ell,q}^2(t,\xi)\pa_\be^\al \{\FI-\FP\} u\right\rag
    \leq \eta \left\|\nu^{1/2}\pa^\al \{\FI-\FP\} u\right\|^2\\
    +C_\eta \left(1+\|\na_x\phi\|_{H^N}^2\right)\left\{\sum_{|\al|\leq N}\left\|\nu^{1/2}\pa^\al \{\FI-\FP\} u\right\|^2+\left\|\na_x (a,b,c)\right\|_{H^{N-1}}^2\right\},
\end{multline*}
for $\eta>0$. Therefore, by plugging all the estimates above into \eqref{lem.ee.p7},
taking summation over $\{|\be|=m,|\al|+|\be|\leq N\}$ for each given $1\leq m\leq N$ and then taking the proper linear combination of those $N-1$ estimates with properly chosen constants  $C_m>0$ $(1\leq m\leq N)$ and $\eta>0$ small enough, one has
\begin{multline}\label{lem.ee.p8}
% \nonumber to remove numbering (before each equation)
 \frac{1}{2}\frac{d}{dt}\sum_{m=1}^NC_m\sum_{
|\al|+|\be|\leq N, |\be|= m}\left\|w_{|\be|-\ell,q}(t,\xi)\pa_\be^\al \{\FI-\FP\} u\right\|^2 \\
+\ka \sum_{
|\al|+|\be|\leq N,|\be|\geq 1}\left\|\nu^{1/2}w_{|\be|-\ell,q}(t,\xi)|\pa_\be^\al \{\FI-\FP\} u\right\|^2\\
+ \frac{\ka}{(1+t)^{1+\vth}}\sum_{|\al|+|\be|\leq N,|\be|\geq 1}\left\|\lag \xi\rag w_{|\be|-\ell,q}(t,\xi)\pa_\be^\al \{\FI-\FP\} u\right\|^2\\
\leq C\de \CD_{N,\ell,q}(t)+C\left\{\sum_{|\al|\leq N}\left\|\nu^{1/2}w_{-\ell,q}(t,\xi)\pa^\al \{\FI-\FP\} u\right\|^2+\left\|\na_x (a,b,c)\right\|_{H^{N-1}}^2\right\}.
\end{multline}

\medskip

\noindent{\bf Step 3.} {\it The proof of \eqref{lem.ee.1}}: We take a proper linear combination of those estimates obtained in the previous two steps as follows. First of all, since $\de>0$ is sufficiently small, the linear combination $M_1\times[\eqref{lem.ee.p1}+\eqref{lem.ee.p2}]+\eqref{lem.ma.2}$ for $M_1>0$ large enough gives
\begin{multline}\label{lem.ee.p9}
\frac{d}{dt}\left\{\frac{M_1}{2}\left[\sum_{|\al|\leq N}\left(\left\|\pa^\al u\right\|^2+\left\|\pa^\al \na_x\phi\right\|^2\right)-\int_{\R^3} |b|^2(a+2c)\,dx\right]+\CE_N^{\rm int}(t)\right\}\\
+\kappa \left\{\sum_{|\al|\leq N}\left\|\nu^{1/2}\pa^\al \{\FI-\FP\} u\right\|^2+\|a\|^2+\left\|\na_x (a,b,c)\right\|_{H^{N-1}}^2\right\}
\leq C\de \CD_{N,\ell,q}(t),
\end{multline}
where the corresponding energy functional is equivalent to $\|u\|_{L^2_\xi(H^N_x)}^2+\|\na_x\phi\|_{H^N}^2$. Now, the further linear combination $M_3\times [M_2\times \eqref{lem.ee.p9}+\eqref{lem.ee.p4}+\eqref{lem.ee.p6}]+\eqref{lem.ee.p8}$
for $M_2>0$ and $M_3>0$ large enough in turn gives
\begin{equation*}
  \frac{d}{dt} \CE_{N,\ell,q}(t)+\kappa \CD_{N,\ell,q}(t)\leq C\de  \CD_{N,\ell,q}(t),
\end{equation*}
where $\CE_{N,\ell,q}(t)$ is given by
\begin{multline}\label{lem.ee.p10}
 \CE_{N,\ell,q}(t)\\
 =M_3\left[M_2\left\{\frac{M_1}{2}\left[\sum\limits_{|\al|\leq N}\left(\left\|\pa^\al u\right\|^2+\left\|\pa^\al \na_x\phi\right\|^2\right)-{\displaystyle\int_{\R^3}} |b|^2(a+2c)\,dx\right]+\CE_N^{\rm int}(t)\right\}\right.\\
\left.+\left\| w_{-\ell,q}(t,\xi) \{\FI-\FP\} u\right\|^2+\sum\limits_{1\leq |\al|\leq N}\left\|w_{-\ell,q}(t,\xi)\pa^\al u\right\|^2\right]\\
 \quad+\sum\limits_{m=1}^NC_m\sum\limits_{
|\al|+|\be|\leq N, |\be|= m}\left\|w_{|\be|-\ell,q}(t,\xi)\pa_\be^\al \{\FI-\FP\} u\right\|^2.
\end{multline}
Notice \eqref{lem.ma.1}. It is easy to see that
$$
 \CE_{N,\ell,q}(t)\sim \trn u \trn_{N,\ell,q}^2 (t).
$$
Therefore, \eqref{lem.ee.1} follows from the above inequality since $\de>0$ is  small enough. This completes the proof of Lemma \ref{lem.ee}.
\end{proof}

\section{Global existence}

Recall \eqref{def.x} for $ X_{N,\ell,q}(t)$. To close the energy estimates under the {\it a priori} assumption \eqref{apa}, one has to obtain the time-decay of $\CE_{N,\ell-1,q}(t)$ and $\|\na_x^2\phi\|_{H^{N-1}}^2$.  The key is to prove the following

\begin{lemma}\label{lem.X}
Assume $-2\leq \ga<0$ and $\int_{\R^3}a_0(x)\,dx=0$. Fix parameters $N$, $\ell_0$, $\ell$ and $q(t)=q+\frac{\la}{(1+t)^\vth}$ as stated in Theorem \ref{thm.m}. Suppose that the {\it a priori} assumption \eqref{apa} holds true for $\de>0$ small enough. Then, one has
\begin{equation}\label{lem.X.2}
     X_{N,\ell,q}(t)\leq C\left\{\eps_{N,\ell,q}^2 +  X_{N,\ell,q}(t)^2\right\},
\end{equation}
for any $t\geq 0$, where $\eps_{N,\ell,q}$ is defined by
\begin{equation}\label{lem.X.3}
   \eps_{N,\ell,q}=\sum_{|\al|+|\be|\leq N}\left\|w_{|\be|-\ell,q}(0,\xi) \pa_\be^\al u_0\right\|+\left\|(1+|x|+\lag \xi\rag^{-\frac{\ga\ell_0}{2}} )u_0\right\|_{Z_1}.
\end{equation}

\end{lemma}

As preparation, we need an additional lemma concerning the time-decay estimates on the macroscopic quantity $(a,b,c)$ in terms of initial data and $X_{N,\ell}(t)$. Here and afterwards, for notational simplicity, we write $X_{N,\ell}(t)=X_{N,\ell,q}(t)$ when $q=0$, and in a similar way, $\eps_{N,\ell}=\eps_{N,\ell,0}$ and $\CE_{N,\ell}(t)=\CE_{N,\ell,0}(t)$ will be used.

\begin{lemma}\label{lem.made}
Under the assumptions of Lemma \ref{lem.X}, one has
\begin{equation}\label{lem.made.1}
  \|u(t)\|+\|\na_x\phi(t)\|\leq C(1+t)^{-\frac{3}{4}}\left \{\eps_{N,\ell}+X_{N,\ell}(t)\right\},
\end{equation}
and
\begin{equation}\label{lem.made.2}
  \|\na_xu(t)\|_{L^2_\xi(H^{N-2}_x)}+\left\|\na_x^2\phi(t)\right\|\leq C(1+t)^{-\frac{5}{4}} \{\eps_{N,\ell}+X_{N,\ell}(t)\},
\end{equation}
for any $0\leq t <T$.
\end{lemma}

\begin{proof}
By Duhamel's principle, the solution $u$ to the Cauchy problem \eqref{vpb1}-\eqref{vpb3} of the nonlinear VPB system can be written as the mild form
\begin{equation*}
    u(t)=e^{t\FB}u_0 +\int_0^t e^{(t-s)\FB} G(s)\,ds
\end{equation*}
with $G=\frac{1}{2} \xi\cdot \na_x\phi u-\na_x\phi \cdot \na_\xi u +\Ga(u,u)$. Notice that $\FP_0 G(t)\equiv 0$ for all $t\geq 0$ and the condition \eqref{thm.lide.1} holds. Applying Theorem \ref{thm.lide}, one has
\begin{multline}\label{lem.made.p1}
\|u(t)\|+\|\na_x\phi(t)\|\\
\leq C(1+t)^{-\frac{3}{4}} \left\{\left\|\lag \xi\rag^{-\frac{\ga}{2}\ell_0}u_0\right\|_{Z_1}+\left\|\lag \xi \rag^{-\frac{\ga}{2}\ell_0} u_0\right\|+\|(1+|x|) a_0\|_{L^1_x}\right\}\\
+C \int_0^t (1+t-s)^{-\frac{3}{4}} \left\{\left\|\lag \xi \rag^{-\frac{\ga}{2}\ell_0} G(s)\right\|_{Z_1}+\left\|\lag \xi \rag^{-\frac{\ga}{2}\ell_0} G(s)\right\|\right\}\,ds
\end{multline}
and
\begin{multline}\label{lem.made.p2}
 \|\na_xu(t)\|_{L^2_\xi(H^{N-2}_x)}+\|\na_x^2\phi(t)\|\\
\leq C(1+t)^{-\frac{5}{4}} \left\{\left\|\lag \xi\rag^{-\frac{\ga}{2}\ell_0}u_0\right\|_{Z_1}+\left\|\lag \xi \rag^{-\frac{\ga}{2}\ell_0} \na_xu_0\right\|_{L^2_\xi(H^{N-2}_x)}+\|(1+|x|) a_0\|_{L^1_x}\right\}\\
+C \int_0^t (1+t-s)^{-\frac{5}{4}} \left\{\left\|\lag \xi \rag^{-\frac{\ga}{2}\ell_0} G(s)\right\|_{Z_1}+\left\|\lag \xi \rag^{-\frac{\ga}{2}\ell_0} \na_x G(s)\right\|_{L^2_\xi(H^{N-2}_x)}\right\}\,ds,
\end{multline}
where we recall that $\ell_0>2\si_1=2(\frac{3}{4}+\frac{1}{2})=5/2$ is a constant. In what follows, we shall estimate the right-hand time integrals in the above two inequalities.

To do that, we {\it claim} that
\begin{equation}\label{lem.made.p3}
  \left\|\lag \xi \rag^{-\frac{\ga}{2}\ell_0} G(t)\right\|_{Z_1}+\sum_{|\al|\leq N-1}\left\|\lag \xi \rag^{-\frac{\ga}{2}\ell_0} \pa^\al G(t)\right\|\leq C \CE_{N,\ell-1}(t),
\end{equation}
for any $0\leq t<T$. Here, recall  $\ell\geq  1+\max\left\{N, \frac{\ell_0}{2}-\frac1\ga\right\}$.
In fact, one has the following estimates.

\medskip

\noindent{\it Estimate on terms containing $\phi$:} Direct calculations yield
\begin{multline}\label{lem.made.p4}
\left\|\lag \xi \rag^{-\frac{\ga}{2}\ell_0} \left\{-\na_x\phi\cdot \na_\xi u +\frac{1}{2}\xi\cdot \na_x\phi u\right\}\right\|_{Z_1}\\
\leq C\left\|\lag \xi \rag^{-\frac{\ga}{2}\ell_0} \|\na_x \phi\|_{L^2_x} \left\{\|\na_\xi u\|_{L^2_x}+\lag \xi \rag \|u\|_{L^2_x}\right\}\right\|_{L^2_\xi}\\
\leq C \left\|\na_x\phi\right\|\left\{\left\|\lag \xi \rag^{-\frac{\ga}{2}\ell_0} \na_\xi u\right\| +\left\|\lag \xi \rag^{-\frac{\ga}{2}\ell_0 +1}u \right\|\right\}\\
\leq C  \|\na_x\phi\|\left\{\left\|w_{-\frac{1}{2}\ell_0}(t,\xi)\na_\xi u\right\|+\left\|w_{-\frac{\ell_0}{2}+\frac{1}{\ga}}(t,\xi) u\right\|\right\}.
\end{multline}
Here, by recalling the assumption $\ell-1>\frac{\ell_0}{2}-\frac{1}{\ga}$, \eqref{lem.made.p4} is further bounded by $C \CE_{N,\ell-1}(t)$.

Similarly, by applying $L^\infty_x$-norm to the low-order derivative terms and using Sobolev's inequality, one also has
\begin{multline*}
\sum_{|\al|\leq N-1}\left\|\lag \xi \rag^{-\frac{\ga}{2}\ell_0} \pa^\al\left\{-\na_x\phi\cdot \na_\xi u +\frac{1}{2}\xi\cdot \na_x\phi u\right\}\right\|\\
\leq C \|\na_x\phi\|_{H^{N}} \sum_{|\al|+|\be|\leq N} \left\|w_{|\be|-(\ell-1)}(t,\xi)\pa_\be^\al u\right\|\leq C\CE_{N,\ell-1}(t).
\end{multline*}

\medskip

\noindent{\it Estimate on terms containing $\Ga(u,u)$:} First consider
\begin{equation}\label{lem.made.p5-0}
   \left \|\lag \xi \rag^{-\frac{\ga}{2}\ell_0} \Ga(u,u)\right\|_{Z_1}\leq C\sum_{\pm}\left\|\lag \xi \rag^{-\frac{\ga}{2}\ell_0} \Ga^\pm\left(\|u\|_{L^2_x},\|u\|_{L^2_x}\right)\right\|_{L^2_\xi},
\end{equation}
where H\"{o}lder's inequality has been used. For the loss term, by noticing
\begin{multline*}
\Ga^-\left(\|u\|_{L^2_x},\|u\|_{L^2_x}\right)=\|u(\xi)\|_{L^2_x}\iint |\xi-\xi_\ast|^\ga q_0(\theta) \FM^{\frac{1}{2}}(\xi_\ast) \|u(\xi_\ast)\|_{L^2_x} d\xi_\ast d\om\\
\leq C \|u(\xi)\|_{L^2_x}\cdot \lag \xi \rag^\ga \sup_{\xi_\ast} \|u(\xi_\ast)\|_{L^2_x},
\end{multline*}
one has
\begin{multline}\label{lem.made.p6}
\left\|\lag \xi \rag^{-\frac{\ga}{2}\ell_0} \Ga^-\left(\|u\|_{L^2_x},\|u\|_{L^2_x}\right)\right\|_{L^2_\xi}\leq C \sup_{\xi_\ast} \|u(\xi_\ast)\|_{L^2_x}\cdot\left\|\lag \xi \rag^{\ga (-\frac{\ell_0}{2}+1)}u\right\|_{L^2_{x,\xi}}\\
\leq C \left(\|\na_\xi u\|+\|\na_\xi^2 u\|\right)\left\|w_{-\frac{\ell_0}{2}+1}(t,\xi) u\right\|\leq C\CE_{N,\ell-1}(t).
\end{multline}
For the gain term,
\begin{multline*}
J_1(t,\xi):=\lag \xi \rag^{-\frac{\ga}{2}\ell_0} \Ga^+\left(\|u\|_{L^2_x},\|u\|_{L^2_x}\right)\\
=\lag \xi \rag^{-\frac{\ga}{2}\ell_0}\iint |\xi-\xi_\ast|^\ga q_0(\theta) \FM^{\frac{1}{2}}(\xi_\ast)  \|u(\xi_\ast')\|_{L^2_x}\|u(\xi')\|_{L^2_x}\,d\xi_\ast d\om.
\end{multline*}
We then consider the following three cases:

\medskip
\noindent\underline{In $\D_1=\{|\xi_\ast|\geq \frac{1}{2}|\xi|\}$}:
\begin{equation*}
  J_1(t,\xi)\leq  \lag \xi \rag^{-\frac{\ga}{2}\ell_0}\int_{\xi_\ast,\om} |\xi-\xi_\ast|^\ga q_0(\theta) \FM^{\frac{q'}{2}}(\xi_\ast) \FM^{\frac{q'}{2}}(\xi)  \|u(\xi_\ast')\|_{L^2_x}\|u(\xi')\|_{L^2_x}.
\end{equation*}
From H\"{o}lder's inequality, as before, $J_1(t,\xi)$ is bounded by
\begin{equation*}
   C\lag \xi \rag^{-\frac{\ga}{2}\ell_0+\frac{1}{2}\ga}\left\{\int_{\xi_\ast,\om} |\xi-\xi_\ast|^\ga q_0(\theta) \FM^{\frac{q'}{2}}(\xi_\ast) \FM^{\frac{q'}{2}}(\xi)  \|u(\xi_\ast')\|_{L^2_x}^2\|u(\xi')\|_{L^2_x}^2\right\}^{\frac{1}{2}},
\end{equation*}
so that
\begin{multline*}
\|J_1(t)\chi_{\D_1}\|_{L^2_\xi}^2\leq C\int_{\xi,\xi_\ast,\om}  |\xi-\xi_\ast|^\ga q_0(\theta)\FM^{\frac{q''}{2}}(\xi_\ast) \FM^{\frac{q''}{2}}(\xi)  \|u(\xi_\ast')\|_{L^2_x}^2\|u(\xi')\|_{L^2_x}^2\\
\leq C \int_{\xi,\xi_\ast}  |\xi-\xi_\ast|^\ga\FM^{\frac{q''}{2}}(\xi_\ast) \FM^{\frac{q''}{2}}(\xi)  \|u(\xi_\ast)\|_{L^2_x}^2\|u(\xi)\|_{L^2_x}^2\\
\leq C \sup_{\xi_\ast} \|u(\xi_\ast)\|_{L^2_x}^2\cdot \int_{x,\xi}|u(x,\xi)|^2\leq C \CE_{N,\ell-1}(t)^2.
\end{multline*}
\underline{In $\D_2=\{|\xi_\ast|\leq \frac{1}{2}|\xi|,|\xi|\leq 1\}$}: From H\"{o}lder's inequality, one has
\begin{equation*}
   J_1(t,\xi)\chi_{\D_2}\leq \left\{\int_{\xi,\om}\chi_{\D_2} |\xi-\xi_\ast|^\ga q_0(\theta) \FM^{\frac{1}{2}}(\xi_\ast) \|u(\xi_\ast')\|_{L^2_x}^2\|u(\xi')\|_{L^2_x}^2\right\}^{\frac{1}{2}}.
\end{equation*}
In $\D_2$, $|\xi-\xi_\ast|\geq \frac{1}{2}|\xi|$ and moreover,
\begin{equation*}
    |\xi'|\leq 2|\xi|+|\xi_\ast|\leq \frac{5}{2}|\xi|,\quad |\xi_\ast'|\leq 2|\xi_\ast|+|\xi|\leq 2|\xi|.
\end{equation*}
These imply that in  $\D_2$,
\begin{equation*}
   |\xi-\xi_\ast|^\ga\leq C \min\{|\xi'|^\ga,|\xi_\ast'|^\ga\} \chi_{|\xi'|\leq 3,|\xi_\ast'|\leq 3 }.
\end{equation*}
Then, using the above estimate and then taking change of variable $(\xi,\xi_\ast)\to (\xi',\xi_\ast')$, one has
\begin{multline*}
\|J_1(t)\chi_{\D_2}\|_{L^2_\xi}^2\leq C \int_{\xi,\xi_\ast,\om}\chi_{\D_2}  |\xi-\xi_\ast|^\ga q_0(\theta) \FM^{\frac{1}{2}}(\xi_\ast)    \|u(\xi_\ast')\|_{L^2_x}^2\|u(\xi')\|_{L^2_x}^2\\
\leq C \int_{\xi,\xi_\ast}\chi_{|\xi|\leq 3,|\xi_\ast|\leq 3}\min\{|\xi|^\ga,|\xi_\ast|^\ga\} \|u(\xi_\ast)\|_{L^2_x}^2\|u(\xi)\|_{L^2_x}^2\\
\leq C \sup_{\xi}\|u(\xi)\|_{L^2_x}^2\cdot \|u\|^2\leq C \CE_{N,\ell-1}(t)^2.
\end{multline*}
\underline{In $\D_3=\{|\xi_\ast|\leq \frac{1}{2}|\xi|,|\xi|\geq 1\}$}: In this case, as in the proof \eqref{lem.ga2.p2-0} for $I_{2,123}^+$, from H\"{o}lder's inequality, we can deduce that
 $J_1(t,\xi)$ is bounded by
\begin{equation*}
   C\lag \xi \rag^{-\frac{\ga}{2}\ell_0}\left\{\int_{\xi_\ast,\om} \left(1+|\xi|^2+|\xi_\ast|^2\right)^\ga q_0(\theta) \FM^{\frac{1}{2}}(\xi_\ast) \|u(\xi_\ast')\|_{L^2_x}^2\|u(\xi')\|_{L^2_x}^2\right\}^{\frac{1}{2}},
\end{equation*}
which further implies that $\|J_1(t)\chi_{\D_3}\|_{L^2_\xi}^2$ is bounded by
\begin{multline*}
C \int_{\xi,\xi_\ast,\om} \lag \xi\rag^{-\ga \ell_0}  \left(1+|\xi|^2+|\xi_\ast|^2\right)^\ga q_0(\theta) \FM^{\frac{1}{2}}(\xi_\ast) \|u(\xi_\ast')\|_{L^2_x}^2\|u(\xi')\|_{L^2_x}^2\\
\leq C \int_{\xi,\xi_\ast,\om} \left\{\lag \xi'\rag^{-\ga \ell_0}+\lag \xi_\ast'\rag^{-\ga \ell_0}\right\}  \left(1+|\xi|^2+|\xi_\ast|^2\right)^\ga q_0(\theta) \FM^{\frac{1}{2}}(\xi_\ast) \|u(\xi_\ast')\|_{L^2_x}^2\|u(\xi')\|_{L^2_x}^2\\
\leq C\int_{\xi,\xi_\ast} \left\{\lag \xi\rag^{-\ga \ell_0}+\lag \xi_\ast\rag^{-\ga \ell_0}\right\}  \left(1+|\xi|^2+|\xi_\ast|^2\right)^\ga \|u(\xi_\ast)\|_{L^2_x}^2\|u(\xi)\|_{L^2_x}^2.
\end{multline*}
Thus, it follows that
\begin{multline*}
\|J_1(t)\chi_{\D_3}\|_{L^2_\xi}^2\leq C \int \lag \xi\rag^{-\ga \ell_0+\ga}\|u(\xi)\|_{L^2_x}^2d\xi \cdot \int \|u(\xi_\ast)\|_{L^2_x}^2d\xi_\ast\\
+C \int \lag \xi_\ast\rag^{-\ga \ell_0+\ga}\|u(\xi_\ast)\|_{L^2_x}^2d\xi_\ast \cdot \int \|u(\xi)\|_{L^2_x}^2d\xi\\
\leq C \left\|w_{-\frac{\ell_0}{2}+\frac{1}{2}}(t,\xi) u\right\|^2\|u\|^2\leq C \CE_{N,\ell-1}(t)^2.
\end{multline*}
Collecting the above estimates,  one has the estimate on the gain term
\begin{equation*}
    \left\|\lag \xi \rag^{-\frac{\ga}{2}\ell_0} \Ga^+\left(\|u\|_{L^2_x},\|u\|_{L^2_x}\right)\right\|_{L^2_\xi}\leq C\|J_1(t)\|_{L^2_\xi}\leq C \CE_{N,\ell-1}(t).
\end{equation*}
Combining the above with \eqref{lem.made.p6} and recalling \eqref{lem.made.p5-0}, we have
\begin{equation}\label{lem.made.p7}
     \left\|\lag \xi \rag^{-\frac{\ga}{2}\ell_0} \Ga(u,u)\right\|_{Z_1}\leq C\CE_{N,\ell}(t).
\end{equation}
Now, we consider the $L^2$-norm of $\lag \xi \rag^{-\frac{\ga}{2}\ell_0}  \pa^\al\Ga(u,u)$ with $|\al|\leq N-1$. Notice
\begin{multline*}
\sum_{|\al|\leq N-1}\left\|\lag \xi \rag^{-\frac{\ga}{2}\ell_0} \pa^\al\Ga(u,u)\right\|
\leq C \sum_{|\al|\leq N-1}\sum_{\al_1+\al_2=\al}\sum_\pm \left\|\lag \xi \rag^{-\frac{\ga}{2}\ell_0} \Ga^\pm\left(\pa^{\al_1}u,\pa^{\al_2}u\right)\right\| \\
\leq C \sum_{|\al|\leq N-1}\sum_{\al_1+\al_2=\al, |\al_1|\leq N/2}\sum_\pm \left\|\lag \xi \rag^{-\frac{\ga}{2}\ell_0} \Ga^\pm\left(\|\na_x\pa^{\al_1}u\|_{H^1_x},\|\pa^{\al_2}u\|_{L^2_x}\right)\right\|_{L^2_\xi}\\
+C \sum_{|\al|\leq N-1}\sum_{\al_1+\al_2=\al, |\al_1|> N/2}\sum_\pm \left\|\lag \xi \rag^{-\frac{\ga}{2}\ell_0} \Ga^\pm\left(\|\pa^{\al_1}u\|_{L^2_x},\|\na_x\pa^{\al_2}u\|_{H^1_x}\right)\right\|_{L^2_\xi},
\end{multline*}
where Sobolev's inequality $\|f\|_{L^\infty_x}\leq C\|\na_x f\|_{H^1_x}$ has been used. Then, as for \eqref{lem.made.p5-0}, one has
\begin{equation*}
  \sum_{|\al|\leq N-1}\left\|\lag \xi \rag^{-\frac{\ga}{2}\ell_0} \pa^\al\Ga(u,u)\right\|\leq C   \CE_{N,\ell-1}(t).
\end{equation*}
This together with \eqref{lem.made.p4}, \eqref{lem.made.p5-0} and \eqref{lem.made.p7} give \eqref{lem.made.p3}. \qed

\medskip

Hence, \eqref{lem.made.p3} together with \eqref{def.x} imply
\begin{multline*}
  \left\|\lag \xi \rag^{-\frac{\ga}{2}\ell_0} G(s)\right\|_{Z_1}+\sum_{|\al|\leq N-1}\left\|\lag \xi \rag^{-\frac{\ga}{2}\ell_0} \pa^\al G(s)\right\|\leq C   \CE_{N,\ell-1}(s)\\
  \leq C(1+s)^{-\frac{3}{2}} \sup_{0\leq s\leq t}(1+s)^{\frac{3}{2}}\CE_{N,\ell-1}(s)\leq C  (1+s)^{-\frac{3}{2}}
  X_{N,\ell}(t),
\end{multline*}
for any $0\leq s\leq t<T$.
Further plugging the above estimate into \eqref{lem.made.p1} and \eqref{lem.made.p2} yields \eqref{lem.made.1} and \eqref{lem.made.2}, respectively, where the following inequalities
\begin{eqnarray*}
% \nonumber to remove numbering (before each equation)
 && \int_0^t (1+t-s)^{-\frac{3}{4}} (1+s)^{-\frac{3}{2}}\,ds\leq C(1+t)^{-\frac{3}{4}},\\
 && \int_0^t (1+t-s)^{-\frac{5}{4}} (1+s)^{-\frac{3}{2}}\,ds\leq C(1+t)^{-\frac{5}{4}},
\end{eqnarray*}
together with \eqref{def.e} and \eqref{lem.ee.p10} for the definition of $\CE_{N,\ell,q}$ have been used.
This completes the proof of Lemma \ref{lem.made}.
\end{proof}

\noindent{\bf Proof of Lemma \ref{lem.X}:} We divide it by three steps. Recall \eqref{def.x} for $ X_{N,\ell,q}(t)$.

\medskip
\noindent{\bf Step 1.} From Lemma \ref{lem.ee},
the time integration of \eqref{lem.ee.1} directly gives
\begin{equation}\label{lem.X.pp4}
  \sup_{0\leq s\leq t}\CE_{N,\ell,q}(s)\leq \CE_{N,\ell,q}(0)\leq C\eps_{N,\ell,q}^2.
\end{equation}
 Here, noticing $\De_x\phi(0,x)=a_0(x)$, we have removed $\|\na_x\phi(0,x)\|_{H^N}^2$ from $\CE_{N,\ell,q}(0)$ by the definition of \eqref{def.tri} or \eqref{def.e} in terms of the inequalities
$$
\|\na_x\phi_0\| \leq C \|a_0\|_{L^1_x}^{2/3}\|a_0\|_{L^2_x}^{1/3}
$$
and
\begin{equation*}
   \left \|\na_x^2\phi_0\right\|_{H^{N-1}}= \left\|\na_x^2\De_x^{-1}a_0\right\|_{H^{N-1}}\leq C \|a_0\|_{H^{N-1}}.
\end{equation*}
Therefore, $\CE_{N,\ell,q}(0)\leq C\eps_{N,\ell,q}^2$ holds true with $\eps_{N,\ell,q}$ given by \eqref{lem.X.3}.

\medskip
\noindent{\bf Step 2.}
Take $0<\eps<1/2$ small enough. Recall \eqref{lem.ee.1}. Notice that \eqref{lem.ee.1} also holds true when $\ell$ is replaced by $\ell-1$ since all the conditions of Lemma \ref{lem.ee} are still satisfied under the assumption that $\ell\geq 1+ N$ and $\sup_{0\leq s<T}X_{N,\ell,q}(s)\leq \de$ with $\de>0$ small enough. Thus, it holds that
\begin{equation*}
     \frac{d}{dt}\CE_{N,\ell-1,q}(t)+\kappa \CD_{N,\ell-1,q}(t)\leq 0.
\end{equation*}
Multiplying the above inequality by $(1+t)^{3/2+\eps}$ gives
\begin{multline}\label{lem.X.pp1}
    \frac{d}{dt}\left[(1+t)^{\frac{3}{2}+\eps} \CE_{N,\ell-1,q}(t)\right]+\kappa (1+t)^{\frac{3}{2}+\eps}\CD_{N,\ell-1,q}(t)\\
    \leq \left(\frac{3}{2}+\eps\right)(1+t)^{\frac{1}{2}+\eps} \CE_{N,\ell-1,q}(t).
\end{multline}
Similarly, from \eqref{lem.ee.1} with $\ell$ replaced by $\ell-1/2$ and further multiplying it by $(1+t)^{1/2+\eps}$, one has
\begin{multline}\label{lem.X.pp2}
\frac{d}{dt}\left[(1+t)^{\frac{1}{2}+\eps} \CE_{N,\ell-\frac{1}{2},q}(t)\right]+\kappa (1+t)^{\frac{1}{2}+\eps}\CD_{N,\ell-\frac{1}{2},q}(t)\\
\leq \left(\frac{1}{2}+\eps\right)(1+t)^{-\frac{1}{2}+\eps} \CE_{N,\ell-\frac{1}{2},q}(t)\leq C\CE_{N,\ell-\frac{1}{2},q}(t).
\end{multline}
Observe from \eqref{def.tri}, \eqref{def.e} and \eqref{def.ed} that
\begin{equation*}
  \CD_{N,\tilde{\ell},q}(t)+ \|(b,c)(t)\|^2+\|\na_x\phi(t)\|^2\geq  \kappa \CE_{N,\tilde{\ell}-\frac{1}{2},q}(t)
\end{equation*}
holds for any given $\tilde{\ell}$. Then,
from taking the time integration over $[0,t]$ of \eqref{lem.X.pp1}, \eqref{lem.X.pp2} and  \eqref{lem.ee.1} and further taking the appropriate linear combination, one has
\begin{multline*}
(1+t)^{\frac{3}{2}+\eps} \CE_{N,\ell-1,q}(t)\leq C \CE_{N,\ell,q}(0)\\
+C\int_0^t (1+s)^{\frac{1}{2}+\eps}\left\{\|(b,c)(s)\|^2+\|\na_x\phi(s)\|^2\right\}\,ds.
\end{multline*}
Here, applying the first estimate \eqref{lem.made.1} in Lemma \ref{lem.made} to the right-hand time integral term of the above inequality and noticing
\begin{equation*}
    \int_0^t  (1+s)^{\frac{1}{2}+\eps}(1+s)^{-\frac{3}{2}}\,ds\leq C(1+t)^\eps,
\end{equation*}
it follows that
\begin{equation*}
    (1+t)^{\frac{3}{2}+\eps} \CE_{N,\ell-1,q}(t)\leq C \CE_{N,\ell,q}(0)+C(1+t)^{\eps} \left\{\eps_{N,\ell}^2+X_{N,\ell}(t)^2\right\},
\end{equation*}
which implies
\begin{equation}\label{ap.td1}
    \sup_{0\leq s\leq t} (1+s)^{\frac{3}{2}} \CE_{N,\ell-1,q}(s)\leq C \left\{\eps_{N,\ell,q}^2+X_{N,\ell}(t)^2\right\}
\end{equation}
holds for any $0\leq t<T$.  Here, $\CE_{N,\ell,q}(0)\leq C\eps_{N,\ell,q}^2$ has been used.

\medskip
\noindent{\bf Step 3.} Notice
\begin{equation*}
   \|\na_x^2\phi \|_{H^{N-1}}^2 \leq C\left\{\|\na_x^2\phi \|^2+\|\na_x a \|_{H^{N-2}}^2\right\}\leq C\left\{\|\na_x^2\phi \|^2+\|\na_x u \|_{L^2_\xi(H^{N-2}_x)}^2\right\}.
\end{equation*}
Then, it is direct to deduce from \eqref{lem.made.2}
\begin{equation}\label{ap.tdh}
\sup_{0\leq s\leq t} (1+s)^{\frac{5}{2}}\|\na_x^2\phi (s)\|_{H^{N-1}}^2\leq C \left\{\eps_{N,\ell}^2+X_{N,\ell}(t)^2\right\}.
\end{equation}
Hence, the desired estimate \eqref{lem.X.2} follows by summing  \eqref{lem.X.pp4}, \eqref{ap.td1} and \eqref{ap.tdh}. This completes the proof of Lemma \ref{lem.X}. \qed

\medskip

Now, we are in a position to give the

\medskip

\noindent{\bf Proof of Theorem \ref{thm.m}:} Given $-2\leq \ga<0$, we fix $N$, $\ell$ and $q(t)=q+\frac{\la}{(1+t)^\vth}$ as stated in Theorem \ref{thm.m}.
The local existence and uniqueness of the solution $u(t,x,\xi)$ to the Cauchy problem \eqref{vpb1}-\eqref{vpb3} can be proved in terms of the energy functional $\CE_{N,\ell,q}(t)$ given by \eqref{def.e}, and the details are omitted for simplicity; see \cite{Guo2}. Then, one only has to obtain the uniform-in-time estimate over $0\leq t<T$ with $0<T\leq \infty$. In fact, by the continuity argument, Lemma \ref{lem.X} implies that under the {\it a priori} assumption \eqref{apa} for $\de>0$ small enough, one has
\begin{equation}\label{thm.g.p1}
X_{N,\ell,q}(t)\leq C \eps_{N,\ell,q}^2, \quad 0\leq t<T,
\end{equation}
provided that $ \eps_{N,\ell,q}$ defined by \eqref{lem.X.3} is sufficiently small. Recalling the condition \eqref{thm.ge.1} for initial data $u_0$ which coincides with \eqref{lem.X.3}, the {\it a priori} assumption \eqref{apa} can be closed. Then, the global existence follows, and \eqref{thm.ge.2} holds true from \eqref{thm.g.p1}, \eqref{def.x}, \eqref{def.e} and \eqref{def.tri}. This completes the proof of Theorem \ref{thm.m}. \qed

\medskip
\noindent
{\bf Acknowledgements:}
The research of the first author was supported partially by the Direct Grant 2010/2011 from CUHK and by the General Research Fund (Project No.~400511) from RGC of Hong Kong.
The research of the second author was supported by the General Research Fund of Hong Kong, CityU No.103108, and the Croucher Foundation. And research of the third author was supported by the grants
from the National Natural Science Foundation of China under
contracts 10871151 and 10925103. This work is also
supported by ``the Fundamental Research Funds for the Central
Universities".

\bigskip

%\newpage
%\addcontentsline{toc}{section}{References}

\end{document}